\documentclass[10pt,twoside,reqno]{amsart}
\usepackage[top=1.4in, bottom=1.4in, left=1in, right=1in]{geometry}
\usepackage{graphicx}
\usepackage{amsmath,latexsym,amssymb}
\usepackage{amsfonts}
\usepackage{mathtools,esint}
\usepackage{amsmath}
\usepackage{amscd,amsthm}
\usepackage{stmaryrd}
\usepackage{fancyhdr}
\usepackage{commath,enumerate}
\usepackage{indentfirst, hyperref}
\usepackage{caption, subcaption}
\usepackage{tikz-cd}

\usepackage{pgf,tikz}
\usepackage{mathrsfs}
\usetikzlibrary{arrows,calc,positioning}
\usepackage{float}
\numberwithin{figure}{section}
\usepackage[bottom]{footmisc}

\numberwithin{equation}{section}
\newtheorem{theorem}{Theorem}[section]
\newtheorem{lemma}[theorem]{Lemma}
\newtheorem{proposition}[theorem]{Proposition}

\theoremstyle{definition}

\newcommand{\Rm}[1]{
  \textup{\uppercase\expandafter{\romannumeral#1}}
}

\newcommand{\C}{\mathbb{C}}

\newcommand{\diff}{\,\mathrm{d}}

\newcommand{\I}{\mathbf{I}}
\newcommand{\n}{\mathbf{n}}
\newcommand{\N}{\mathbb{N}}
\renewcommand{\P}{\mathbf{P}}
\newcommand{\pv}{\mathrm{p.v.}\,}
\newcommand{\px}{\partial_x}
\newcommand{\QQ}{\mathbf{Q}}

\newcommand{\R}{\mathbb{R}}
\newcommand{\Res}{\mathrm{Res}}
\newcommand{\Rs}{\mathscr{R}}
\renewcommand{\t}{\mathbf{t}}
\newcommand{\T}{\mathbb{T}}

\renewcommand{\u}{\mathbf{u}}
\newcommand{\ve}{\varepsilon}
\newcommand{\vp}{\varphi}
\newcommand{\x}{\mathbf{x}}
\newcommand{\Z}{\mathbb{Z}}

\DeclareMathOperator{\hilbert}{\mathbf{H}}
\DeclareMathOperator{\sgn}{\mathrm{sgn}}
\DeclareMathOperator{\supp}{\mathrm{supp}}
\DeclareMathOperator{\nonquin}{\mathcal{N}_{\geq 5}}
\DeclareMathOperator{\noncub}{\mathcal{N}_{\geq 3}}

\renewcommand{\iint}{\int}

\author{John K. Hunter}
\address{Department of Mathematics, University of California at Davis}
\email{jkhunter@ucdavis.edu}
\thanks{JKH was supported by the NSF under grant numbers DMS-1616988 and DMS-1908947}
\author{Ryan C. Moreno-Vasquez}
\address{Department of Mathematics, University of California at Davis}
\email{rcmorenovasquez@math.ucdavis.edu}
\author{Jingyang Shu}
\address{Department of Mathematics, Temple University}
\email{jyshu@temple.edu}
\author{Qingtian Zhang}
\address{Department of Mathematics, West Virginia University}
\email{qingtian.zhang@mail.wvu.edu}
\title[Euler fronts and the BH equation]{On the approximation of vorticity fronts by the Burgers-Hilbert equation}
\date{\today}

\begin{document}

\begin{abstract}
This paper proves that the motion of small-slope vorticity fronts in the two-dimensional incompressible Euler equations is approximated
on cubically nonlinear timescales by a Burgers-Hilbert equation derived by Biello and Hunter (2010) using formal asymptotic expansions. The proof uses a modified energy method to show that the contour dynamics equations for vorticity fronts in the Euler equations and the Burgers-Hilbert equation are both approximated by the same cubically nonlinear asymptotic equation. The contour dynamics equations for Euler vorticity fronts are also derived.
\end{abstract}

\maketitle

\section{Introduction}

The two-dimensional incompressible Euler equations have solutions for vorticity fronts located at $y=\vp(x,t)$ that separate
two regions with distinct, constant vorticities $- \alpha_+$ and $- \alpha_-$ in $y > \vp(x,t)$ and $y < \vp(x,t)$, respectively.
As illustrated in Figure~\ref{fig:sol}, these solutions may be regarded as perturbations of a piecewise linear shear flow
$(U(y),0)$ with
\begin{equation}
U(y) = \begin{cases} \alpha_+ y & \text{if $y > 0$,} \\ \alpha_- y &\text{if $y<0$.}\end{cases}
\label{shear_flow}
\end{equation}
In his studies of the stability of shear flows, Rayleigh showed that the flow \eqref{shear_flow}
is linearly stable \cite{Ray80}. He also showed that the vorticity front supports unidirectional waves and computed the Fourier expansion of a spatially periodic traveling wave on the front up to fifth order in the slope of the front \cite{Ray95}.  Rayleigh did not, however, consider the
more complex nonlinear dynamics of small-slope fronts with general spatial profiles that are described by the equations analyzed here.

\begin{figure}[h]
\centering
\includegraphics[width=0.6\textwidth]{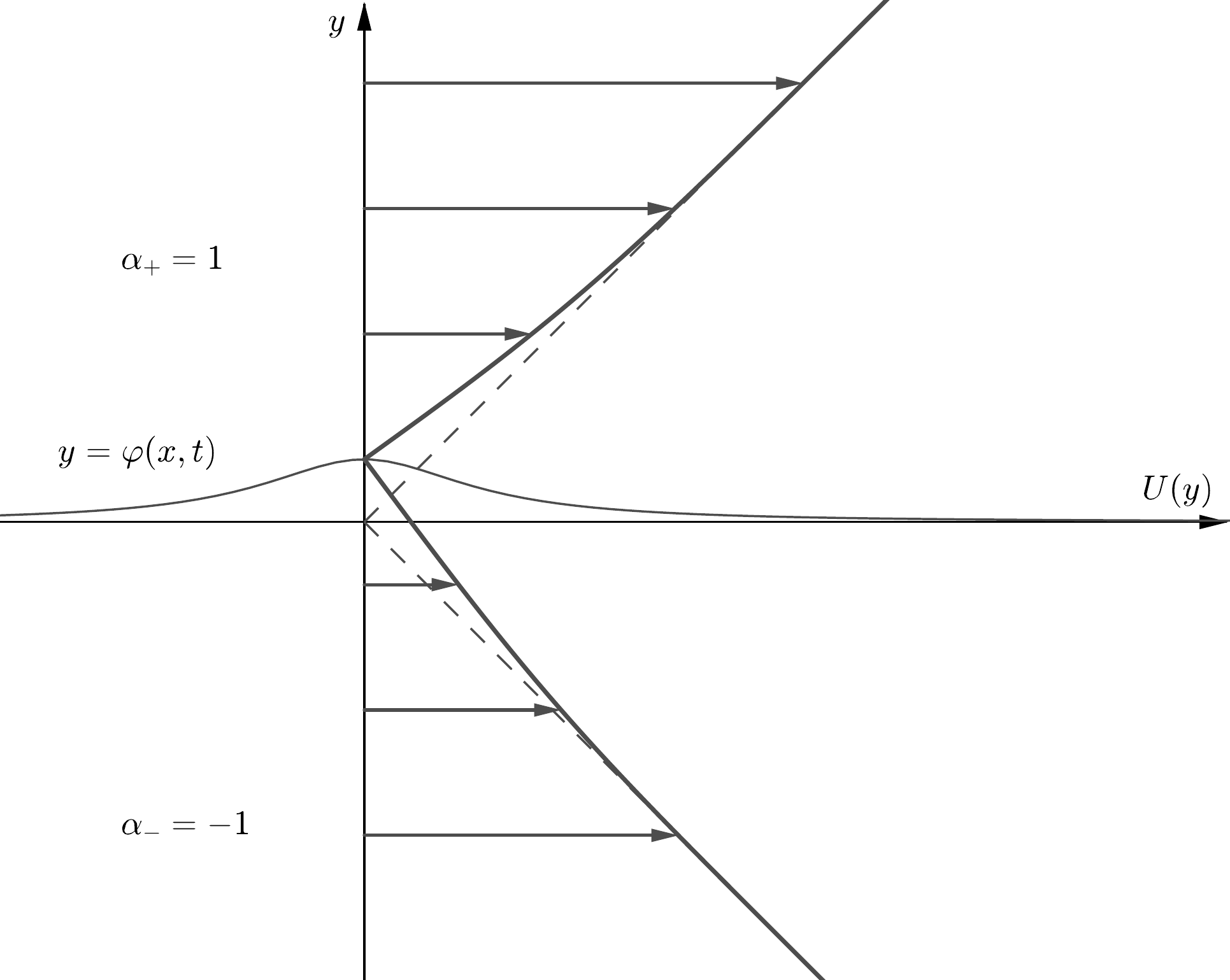}
\caption{An illustration of the $x$-velocity field of a vorticity front in a fluid with symmetric constant vorticities ($\alpha_+ = - \alpha_-$). The dashed line is the unperturbed shear flow \eqref{shear_flow} plotted versus $y$, and the solid line is the perturbed $x$-velocity. The influence of the front motion on the velocity field decays as $|y|\to\infty$.}\label{fig:sol}
\end{figure}

We non-dimensionalize the time variable $t$ so that
\[
\frac{\alpha_+ - \alpha_-}{2} = 1.
\]
Then, in Appendix~\ref{sec:contour}, we show that the displacement $\vp(x,t)$ of the vorticity front  satisfies the following evolution equation
\begin{align}
\label{cde}
\vp_t(x, t) + \frac{m}{2} \px \left[\vp^2(x, t)\right] + \frac{1}{2 \pi} \int_\R \left[\vp_x(x, t) - \vp_x(x + \zeta, t)\right] \log\bigg[1 + \frac{[\vp(x, t) - \vp(x + \zeta, t)]^2}{|\zeta|^2}\bigg] \diff{\zeta} = \hilbert[\vp](x, t),
\end{align}
where $\hilbert$ denotes the Hilbert transform with respect to $x$,  which is a Fourier multiplier operator with symbol $- i \sgn{\xi}$, and
\[
m = \frac{\alpha_+ + \alpha_-}{\alpha_+ - \alpha_-}.
\]

In this paper, we prove that small-slope solutions of \eqref{cde} are approximated on cubically nonlinear time scales by solutions of the following Burgers-Hilbert equation (see Theorem~\ref{thm:euler-bh})
\begin{align}
\label{bh}
u_t(x, t) + \frac{\sqrt{m^2 + 1}}{2} \px \left[u^2(x, t)\right] = \hilbert[u](x, t).
\end{align}
Moreover, we prove that small-slope solutions of both \eqref{cde} and \eqref{bh} are approximated on cubic time scales by solutions of the following asymptotic equation
\begin{equation}
w_t + \frac{m^2 + 1}{2} \px \left\{w^2 |\px| w - w |\px| w^2 + \frac{1}{3} |\px| w^3\right\} = \hilbert[w],
\label{weq}
\end{equation}
where $|\px| = \hilbert \px$ is the Fourier multiplier with symbol $|\xi|$ (see Theorem~\ref{thm:euler-cub}).
A multiple-scale form of this asymptotic solution is given in \eqref{defw} and \eqref{cub}.

Equations \eqref{bh}--\eqref{weq} were derived previously as descriptions of vorticity fronts in \cite{BH10} by means of formal asymptotic expansions of the Burgers-Hilbert and Euler equations; the present paper provides a proof of that result.
The proof uses a modified energy method introduced in \cite{HITW15} to eliminate the effect of the quadratic terms in \eqref{cde}--\eqref{bh}
on energy estimates for the error between solutions of \eqref{cde}--\eqref{bh} and \eqref{weq} on cubic timescales.

The Burgers-Hilbert description is significant because it gives
a clear picture of the nonlinear dynamics of small-slope vorticity fronts. Solutions of the linearized Burgers-Hilbert equation
$u_t= \hilbert[u]$ oscillate with frequency one between an arbitrary spatial profile, its Hilbert transform, and their negatives.
The oscillating spatial profile of the front then undergoes a slow, alternate compression and expansion due to the
Burgers nonlinearity, leading to a complex deformation of the front profile and an effectively cubic nonlinearity.
Numerical solutions show that wave-breaking in small-slope solutions of the Burgers-Hilbert equation corresponds to the formation of multiple, extraordinarily thin filaments in the vorticity front \cite{BHnum}, similar to the ones observed in vortex patches \cite{Dr1, Dr2}.

A Burgers-Hilbert equation was written down by Marsden and Weinstein  \cite{MW83} as a quadratic truncation for the motion of the boundary of a vortex patch, which, from \eqref{cde}, gives the equation
\[
\vp_t(x, t) + \frac{m}{2} \px \left[\vp^2(x, t)\right] = \hilbert[\vp](x, t).
\]
However, this equation does not provide an approximation for front motions on cubic time scales; for example, in the symmetric case $\alpha_+ = -\alpha_-$, we have $m =0$ and the nonlinear term vanishes in the Burgers-Hilbert equation in \cite{MW83}. Rather, one has to use the appropriately renormalized nonlinear coefficient given in \eqref{bh}. From the point of view of normal forms, when one uses a near identity transformation to remove the quadratic term from \eqref{bh} (which is nonresonant), one gets the same cubic term as the one
that arises from the full Euler front equation \eqref{cde}. Dimensional analysis provides some explanation for why the quadratically nonlinear Burgers-Hilbert equation should provide a description of the cubically nonlinear dynamics of vorticity fronts with small slopes \cite{BH10}.
Further results on the Burgers-Hilbert equation can be found in \cite{BN14, BZ17, CCG10, Hun18, Hur18, KPPVp, KPV20, KV19p, SW, Yang}.

\section{Statement of the main theorems}

For $n\in \N$, we denote by $H^n(\R)$ the standard $L^2$-Sobolev space equipped with norm
\[
\|f\|_{H^n(\R)}^2 = \int_\R |f(x)|^2 \diff{x} + \int_\R | \partial^n f(x)|^2 \diff{x},
\]
and we abbreviate $\int_\R = \int$ when there is no confusion. For simplicity, we restrict our analysis to Sobolev spaces of integer orders.

In order to deal with the Euler front equation \eqref{cde} and the Burgers-Hilbert equation \eqref{bh} simultaneously, we consider the equation
\begin{align}
\label{cde'}
\vp_t(x, t) + \frac{\rho}{2} \px \left[\vp^2(x, t)\right] + \frac{\sigma}{2 \pi} \int_\R \left[\vp_x(x, t) - \vp_x(x + \zeta, t)\right] \log\bigg[1 + \frac{[\vp(x, t) - \vp(x + \zeta, t)]^2}{|\zeta|^2}\bigg] \diff{\zeta} = \hilbert[\vp](x, t),
\end{align}
where $\rho, \sigma \in \R$ are parameters. If $\rho = m$, $\sigma = 1$, then \eqref{cde'} reduces to \eqref{cde}, and if $\rho = \sqrt{m^2 + 1}$, $\sigma = 0$, then \eqref{cde'} reduces to \eqref{bh}.
Local well-posedness of the Cauchy problem for this equation with
\[
\vp\in C([0,T]; H^n(\R))\cap C^1([0,T]; H^{n-1}(\R))
\]
and $n \ge 3$ follows by standard arguments for quasilinear equations \cite{chemin, IT20, MB02, Ta}, using energy estimates similar to the ones in \cite{HSZ19pa}.

Equation \eqref{cde'} has the formal multiple-scale asymptotic solution
\begin{equation*}
\vp(x,t) = \ve e^{t\hilbert} v(x,\ve^2 t) + O(\ve^2)\qquad \text{as $\ve\to 0$ with $t= O(\ve^{-2})$},
\end{equation*}
where $v(x,\tau)$ satisfies
\begin{align}
\label{cub2}
v_\tau + \frac{\rho^2 + \sigma}{2} \px \left\{v^2 |\px| v - v |\px| v^2 + \frac{1}{3} |\px| v^3\right\} = 0.
\end{align}
Local well-posedness of the Cauchy problem for this equation with
\[
\vp\in C([0,T]; H^n(\R))\cap C^1([0,T]; H^{n-1}(\R))
\]
and $n \ge 3$ also follows by standard arguments for quasilinear equations,
using analogous energy estimates on $\R$ to the ones given in \cite{HS18, Ifr12} for spatially periodic solutions on $\T$. Equation \eqref{cub2} has a complex form \eqref{psieqn}, which is what we use when constructing approximate solutions since it simplifies the algebra.

As stated in the next theorem, the leading order formal asymptotic solution
\begin{equation}
w(x,t;\ve) = \ve e^{t\hilbert} v(x,\ve^2 t) = \ve\left[v(x,\ve^2 t) \cos{t} + \hilbert[v](x,\ve^2 t) \sin{t}\right]
\label{defw}
\end{equation}
approximates solutions of \eqref{cde'} over cubic timescales.

\begin{theorem}
\label{thm:euler-cub}
Fix an integer $n \geq 3$ and constants $C, T > 0$. Let $n_v \geq n + 5$. Then there exist constants $C', \ve_0 > 0$ such that for all
$0<\ve < \ve_0$, all solutions $v\in C([0,T], H^{n_v}(\R))$ of \eqref{cub2}, and all $\vp_0\in H^n(\R)$ with
\begin{equation}
\sup_{\tau\in[0,T]} \|v(\cdot,\tau)\|_{H^{n_v}} \le C,\qquad \|\vp_0 - v(\cdot,0)\|_{H^n} \le C\ve,
\label{vests}
\end{equation}
there exists a unique solution $\vp \in C([0,T/\ve^2],H^n(\R))$ of the Cauchy problem for the modified Euler front equation \eqref{cde'}
with initial data $\vp(\cdot,0) = \ve\vp_0$, and this solution satisfies
\begin{align}
\label{solest}
\sup_{t \in [0, T / \ve^2]} \left\|\vp(\cdot, t) - \ve \big[v(\cdot, \ve^2 t) \cos{t} + \hilbert[v](\cdot, \ve^2 t) \sin{t}\big]\right\|_{H^n(\R)} \leq C' \ve^{2}.
\end{align}
For the Burgers-Hilbert equation, \eqref{cde'} with $\sigma =0$, the same result holds for $n\ge 2$.
\end{theorem}

Here, we require some additional regularity and higher-order estimates for the asymptotic solution $v$ in order to construct sufficiently accurate approximate solutions of \eqref{cde'}.
We remark that in the case of the Burgers-Hilbert equation, the existence of small, smooth $H^n$-solutions on some cubic life span is proved in \cite{HI12,HITW15}. However, the previous theorem
shows that the Burgers-Hilbert solution exists and remains close to the asymptotic solution for any time-interval on which the asymptotic
solution exists.
For the modified Euler front equation \eqref{cde'} with $\sigma \neq 0$, we need to assume that $n \geq 3$ in order to estimate an error term $\sigma (J_5 + J_6)$ that appears in Section~\ref{sec:bhest}.

If either  $\rho = m$, $\sigma = 1$ or $\rho = \sqrt{m^2 + 1}$, $\sigma = 0$, then \eqref{cub2} reduces to
\begin{align}
\label{cub}
v_\tau + \frac{m^2 + 1}{2} \px \left\{v^2 |\px| v - v |\px| v^2 + \frac{1}{3} |\px| v^3\right\} = 0,
\end{align}
so \eqref{cde} and \eqref{bh} have the same asymptotic equation. Moreover,
if $v(x,\tau)$ satisfies \eqref{cub}, then the leading order approximation $w(x,t;\ve)$ in \eqref{defw}
satisfies \eqref{weq} (see Lemma~\ref{lem:vwtrans}), so \eqref{weq} provides an unscaled version of the asymptotic equation for both \eqref{cde} and \eqref{bh}.

The next theorem for approximating solutions of the Euler front equation \eqref{cde} by solutions of the Burgers-Hilbert equation \eqref{bh} then follows immediately by comparing $\vp$ and $u$ with the asymptotic solution $w$ with initial data $w(\cdot,0) = \ve v_0$.

\begin{theorem}
\label{thm:euler-bh}
Fix an integer $n \geq 3$ and a constant $C > 0$. Let $v_0 \in H^{n_v}(\R)$ where $n_v \geq n + 5$, and let $T>0$ be an existence time for the solution $v \in C([0,T]; H^{n_v})$ of the asymptotic equation \eqref{cub} with initial data $v(\cdot,0) = v_0$. Then there exist constants $C', \ve_0 > 0$, depending on $\|v_0\|_{H^{n_v}}$ and $T$, such that for all $0 < \ve < \ve_0$ and all $\vp_0, u_0 \in H^n(\R)$ with
\[
\|\vp_0 - v_0\|_{H^n(\R)} + \|u_0 - v_0\|_{H^n(\R)} \leq C \ve,
\]
there exist unique solutions $\vp \in C([0, T / \ve^2]; H^n(\R))$ of \eqref{cde} and $u \in C([0, T / \ve^2]; H^n(\R))$ of \eqref{bh} with initial data $\vp(\cdot, 0) = \ve \vp_0$ and $u(\cdot, 0) = \ve u_0$, respectively,
which satisfy
\[
\sup_{t \in [0, T / \ve^2]} \|\vp(\cdot, t) - u(\cdot, t)\|_{H^n} \leq C \ve^2.
\]
\end{theorem}

The rest of the paper is devoted to the proof of Theorem~\ref{thm:euler-cub}. The main idea of the proof is to use a modified energy inspired by a normal form transformation \cite{HITW15} to obtain cubic energy estimates that do not lose derivatives.
Schneider and Uecker \cite{SU17} give an introduction to this method, and related proofs for NLS approximations can be found in \cite{CW17, Dul17, DH18, IT19, PL19}. Unlike these papers, where the waves under study are dispersive, the Euler front equation \eqref{cde} is non-dispersive with no quadratic three-wave resonances and many cubic four-wave resonances. In particular, the spatial spectrum of $\vp$ is not localized near a specific wavenumber.  This property explains why the asymptotic equation for the Euler front equation is \eqref{cub}, rather than an NLS equation. Moreover, in the absence of dispersive decay,
one does not expect to get the existence of global solutions for general small, smooth initial data as in \cite{CGSI19, HSZ20, HSZ18p}.

In Section~\ref{sec:bhres}, we derive an approximate solution of \eqref{cde'} and obtain residual estimates. In Section~\ref{sec:bherr}, we define a modified energy for the error equation. In Section~\ref{sec:bhest}, we obtain energy estimates for the error and use them to prove Theorem~\ref{thm:euler-cub}. In the appendices we derive the contour dynamics equation for Euler fronts and prove some algebraic details used in the derivation of the asymptotic equation.

Throughout the paper, we use $C$ to denote a constant independent of $\ve$, which may change from line to line,
and the notation $O(\delta)$ denotes a term satisfying $|O(\delta)| \le C \delta$.

\section{Formal approximation and residual estimates}
\label{sec:bhres}

In this section, we construct an approximate solution of \eqref{cde'} and estimate its residual.
We first give an expansion of the nonlinear term in the equation.

\subsection{Expansion of the nonlinearity}
We write \eqref{cde'} as
\begin{align}
\label{cde2}
&\vp_t + \frac{\rho}{2} \px (\vp^2) + \sigma \noncub[\vp] = \hilbert[\vp],
\end{align}
where $\noncub$ is the cubic term
\begin{equation}
\noncub[\vp](x, t) = \frac{1}{2 \pi} \int_\R \left[\vp_x(x, t) - \vp_x(x + \zeta, t)\right] \log\left[1 + \frac{[\vp(x, t) - \vp(x + \zeta, t)]^2}{|\zeta|^2}\right] \diff{\zeta}.
\label{noncub}
\end{equation}
We then have the following expansion.

\begin{lemma}
\label{lem:noncub}
Let $\noncub[\vp]$ be given by \eqref{noncub} where $\|\vp\|_{H^{n + 1}} < 1$ for some integer $n\ge 2$. Then
\begin{align}
\label{N5-1}
\noncub[\vp] =  \frac{1}{2} \px \left\{\vp^2 |\px| \vp - \vp |\px| \vp^2 + \frac{1}{3} |\px| \vp^3\right\} + \nonquin[\vp],
\end{align}
where the quintic and higher-degree terms $\nonquin[\vp]$ satisfy the estimate
\begin{equation}
\label{nonquinest}
\|\nonquin[\vp]\|_{H^{n}} \leq C(n) \sum_{k = 2}^\infty k^{n-1} \|\vp\|_{H^{n+1}}^{2 k + 1}.
\end{equation}
\end{lemma}

\begin{proof}
The assumption $\|\vp\|_{H^{n + 1}} < 1$ guarantees that $\|\vp_x\|_{L^\infty} < 1$, so the mean value theorem implies that
\[
\frac{[\vp(x, t) - \vp(x + \zeta, t)]^2}{|\zeta|^2} < 1 \qquad \forall x, \zeta \in \R, \zeta \neq 0.
\]
We can therefore Taylor expand the right-hand side of \eqref{noncub} to obtain that
\begin{align*}
\noncub[\vp](x, t) = \sum_{k = 1}^\infty \frac{(-1)^{k + 1}}{2 \pi k (2k + 1)} \px \int_\R \frac{[\vp(x, t) - \vp(x + \zeta, t)]^{2 k + 1}}{\zeta^{2 k}} \diff{\zeta}.
\end{align*}

Taking the $n$-th derivative of this equation and using the Leibnitz rule, we see that a general term of degree $2k+1$ in $\px^n \noncub[\vp](x, t)$ has the form
\[
C(n, k) \int_\R \prod_{\ell = 1}^{2 k + 1} \left[\vp^{(j_\ell)}(x, t) - \vp^{(j_\ell)}(x + \zeta, t)\right] \frac{\diff{\zeta}}{|\zeta|^{2 k}},
\]
where $C(n, k) > 0$ is a constant depending on $n$ and $k$ which can be bounded by $C(n) k^{n-1}$, and
\begin{equation}
\label{jcon}
j_1 + j_2 + \dotsb + j_{2k + 1} = n + 1, \qquad 0 \leq j_1 \leq j_2 \leq \dotsb \leq j_{2 k + 1} \leq n + 1.
\end{equation}

Using the Minkowski and H\"{o}lder inequalities, we obtain that
\begin{align*}
& \bigg\|C(n, k) \int_\R \prod_{\ell = 1}^{2 k + 1} \left[\vp^{(j_\ell)}(x, t) - \vp^{(j_\ell)}(x + \zeta, t)\right] \frac{\diff{\zeta}}{|\zeta|^{2 k}}\bigg\|_{L^2_x}\\
\leq ~& C(n, k) \bigg\|\int_{|\zeta| < 1} \left[\vp^{(j_{2 k + 1})}(x, t) - \vp^{(j_{2 k + 1})}(x + \zeta, t)\right]\prod_{\ell = 1}^{2 k} \frac{\vp^{(j_\ell)}(x, t) - \vp^{(j_\ell)}(x + \zeta, t)}{|\zeta|} \diff{\zeta}\bigg\|_{L^2_x}\\
& \qquad + C(n, k) \bigg\|\int_{|\zeta| \ge 1} \prod_{\ell = 1}^{2 k + 1} \left[\vp^{(j_\ell)}(x, t) - \vp^{(j_\ell)}(x + \zeta, t)\right] \frac{\diff{\zeta}}{|\zeta|^{2 k}}\bigg\|_{L^2_x}\\
\leq ~& C(n, k) \|\px^{j_{2 k + 1}} \vp\|_{L^2} \prod_{\ell = 1}^{2 k} \|\px^{j_\ell + 1} \vp\|_{L^\infty} +  C(n, k) \|\px^{j_{2 k + 1}} \vp\|_{L^2} \prod_{\ell = 1}^{2 k} \|\px^{j_\ell} \vp\|_{L^\infty}\\
\leq ~& C(n,k)  \|\vp\|_{H^{n + 1}}^{2 k + 1},
\end{align*}
where the last line holds by Sobolev embedding, since
\eqref{jcon} implies that $0\le j_{\ell}\le n-1$ for $1\le \ell \le 2k$.
Summing over all these terms with $k \ge 2$, and using the bound for $C(n,k)$, we find that
\begin{align*}
\nonquin[\vp](x, t) = \sum_{k = 2}^\infty \frac{(-1)^{k + 1}}{2 \pi k (2k + 1)} \px \int_\R \frac{[\vp(x, t) - \vp(x + \zeta, t)]^{2 k + 1}}{\zeta^{2 k}} \diff{\zeta}.
\end{align*}
satisfies \eqref{nonquinest}.

Finally, as shown in \cite{HS18}, the leading-order cubic term in $\noncub[\vp]$ with $k=1$ can be written as
\begin{equation}
\frac{1}{6 \pi} \px \int_\R \frac{[\vp(x, t) - \vp(x + \zeta, t)]^{3}}{\zeta^{2 }} \diff{\zeta}
=  \frac{1}{2} \px \left\{\vp^2 |\px| \vp - \vp |\px| \vp^2 + \frac{1}{3} |\px| \vp^3\right\},
\label{cubic_term}
\end{equation}
so \eqref{N5-1} follows.
\end{proof}

We remark that there is a cancellation of derivatives on the right-hand side of \eqref{cubic_term}, and the $H^n$-norm of the cubic term
can be estimated in terms of $\|\vp\|_{H^{n+1}}$ as above, but we will not use this fact in the present paper.

\subsection{Approximate solution}

We denote the residual of a function $f(x,t)$ by
\begin{align}
\label{res}
\Res(f) = - f_t - \frac{\rho}{2} \left(f^2\right)_x - \sigma \noncub[f] + \hilbert[f],
\end{align}
which measures the extent to which $f$ fails to satisfy \eqref{cde'}.

\label{sec:approx}
We look for an approximate solution $\vp \approx \ve V$ of  \eqref{cde'} of the form
\begin{align}
\label{V}
\begin{split}
\ve V(x, t; \ve) = \ve V_0(x, t, \ve^2 t) + \ve^2 V_1(x, t, \ve^2 t) + \ve^3 V_2(x, t, \ve^2 t),
\end{split}
\end{align}
where $0 < \ve \ll 1$ is a small parameter and the functions $V_n(x,t,\tau)$ are to be determined.

Using \eqref{V} and \eqref{N5-1} in \eqref{res}, we find that the residual of $\ve V$ is given by
\begin{align*}
\Res(\ve V) & = - \ve \left(V_{0 t} - \hilbert[V_0]\right) - \ve^2 \left(V_{1 t} - \hilbert[V_1] + \frac{\rho}{2} (V_0^2)_x\right)\\
& \quad - \ve^3 \left(V_{0 \tau} + V_{2 t} - \hilbert[V_2] + \rho (V_0 V_1)_x + \frac{\sigma}{2} \px \Big\{V_0^2 |\px| V_0 - V_0 |\px| V_0^2 + \frac{1}{3} |\px| V_0^3\Big\}\right) + O(\ve^4).
\end{align*}
In order to make $\Res(\ve V) = O(\ve^4)$, we require that $V_0$, $V_1$, $V_2$ satisfy
\begin{align}
& V_{0 t} = \hilbert[V_0], \label{ve1'}\\
&V_{1 t} + \frac{\rho}{2} (V_0^2)_x = \hilbert[V_1], \label{ve2'}\\
& V_{2 t} + V_{0 \tau} + \rho (V_0 V_1)_x + \frac{\sigma}{2} \px \Big\{V_0^2 |\px| V_0 - V_0 |\px| V_0^2 + \frac{1}{3} |\px| V_0^3\Big\} = \hilbert[V_2]. \label{ve3'}
\end{align}

It is convenient to use a complex representation for the solutions of these equations. Let $\P$ be the projection onto positive spatial wavenumbers.  If $\I$ denotes the identity operator, then $\P$ and its complement $\QQ = \I - \P$ (the projection onto negative spatial wavenumbers) are given by
\begin{align*}
\P = \frac{\I + i \hilbert}{2}, \qquad \QQ = \frac{\I - i \hilbert}{2}.
\end{align*}

\textbf{Solution for $V_0$.} The solution of \eqref{ve1'} can be written as
\begin{align}
\label{V0}
V_0(x, t, \tau) = \Psi(x, \tau) e^{- i t} + \Psi^*(x, \tau) e^{i t},
\end{align}
where the complex-valued function $\Psi$ satisfies $\P[\Psi] = \Psi$.
In particular, it follows that $\P[\Psi^2] = \Psi^2$ and $\hilbert[\Psi^2] = - i \Psi^2$.

\textbf{Solution for $V_1$.}  A solution of \eqref{ve2'} can be written as
\begin{align}
\label{V1}
V_1(x, t, \tau) = \Psi_{12}(x, \tau) e^{- 2 i t} + \Psi_{10}(x, \tau) + \Psi_{12}^*(x, \tau) e^{2 i t},
\end{align}
where
\begin{align}
\label{ve2a}
\Psi_{12} = - \frac{i \rho}{2} \left(\Psi^2\right)_x \qquad \text{and} \qquad \Psi_{10} = - \rho \hilbert\left[|\Psi|^2\right]_x.
\end{align}
We omit a solution of the homogeneous equation from $V_1$ since we do not need it.

\textbf{Solution for $V_2$.} To proceed further, we use the following proposition, which is proved by a straightforward computation  \cite{BH10}.

\begin{proposition}
\label{prop:sol}
Consider the equation
\begin{align*}
f_t = \hilbert[f] + B(x) e^{- i n t},
\end{align*}
where $n \in \Z$ and $B \in L^2(\R;\C)$. Then:
\begin{enumerate}
\item If $n^2 \neq 1$, then the equation is uniquely solvable for every $B$;
\item If $n = 1$, then the equation is solvable if and only if $\P[B] = 0$;
\item If $n = -1$, then the equation is solvable if and only if $\QQ[B] = 0$.
\end{enumerate}
\end{proposition}

Equation \eqref{ve3'} has solutions of the form
\begin{align}
\label{V2}
V_2(x, t, \tau) = \Psi_{23}(x, \tau) e^{-3 i t} + \Psi_{21}(x, \tau) e^{- i t} + \Psi_{21}^*(x, \tau) e^{i t} + \Psi_{23}^*(x, \tau) e^{3 i t}.
\end{align}
Using \eqref{V0}--\eqref{V2} in \eqref{ve3'}, and equating terms proportional to $e^{-3it}$ and $e^{-it}$, we obtain the following equations for $\Psi_{23}$ and $\Psi_{21}$
\begin{align}
& - 3 i \Psi_{23} + \rho (\Psi \Psi_{12})_x + \frac{\sigma}{2} \px\bigg\{\Psi^2 |\px| \Psi - \Psi |\px| \Psi^2 + \frac{1}{3} |\px| \Psi^3\bigg\} = \hilbert[\Psi_{23}], \label{ve3'23}\\
\begin{split}
& (\Psi_{21} e^{- i t})_t + \Psi_\tau e^{- i t} + \rho \left(\Psi \Psi_{10} + \Psi^* \Psi_{12}\right)_x e^{- i t}\\
& \qquad + \frac{\sigma}{2} \px \bigg\{2 |\Psi|^2 |\px| \Psi + \Psi^2 |\px| \Psi^* - 2 \Psi |\px| |\Psi|^2 - \Psi^* |\px| \Psi^2 + |\px| (\Psi |\Psi|^2)\bigg\} e^{- i t} = \hilbert[\Psi_{21} e^{- i t}].
\end{split} \label{ve3'21}
\end{align}
The solution of \eqref{ve3'23} for $\Psi_{23}$ is given by
\begin{align*}
\Psi_{23} & = - \frac{3 i}{16} \px \left\{2 \rho \Psi \Psi_{12} + \sigma \Psi^2 |\px| \Psi - \sigma \Psi |\px| \Psi^2 + \frac{\sigma}{3} |\px| \Psi^3\right\}\\
& \qquad + \frac{1}{16} |\px| \left\{2 \rho \Psi \Psi_{12} + \sigma \Psi^2 |\px| \Psi - \sigma \Psi |\px| \Psi^2 + \frac{\sigma}{3} |\px| \Psi^3\right\}.
\end{align*}

Applying Proposition~\ref{prop:sol} to \eqref{ve3'21}
and simplifying the result (see Appendix~\ref{sec:simpsolv}), we find that the solvability condition for $\Psi_{21} e^{- i t}$ is satisfied if
\begin{align}
\label{psieqn}
\Psi_\tau = (\rho^2 + \sigma) \P\left[i |\Psi|^2 \Psi_x + \Psi \hilbert[|\Psi|^2]_x\right]_x.
\end{align}
Equation \eqref{psieqn} is the complex form of \eqref{cub2}. Indeed, substituting
\begin{align*}
\Psi = \P[v] = \frac{1}{2} \left[v + i \hilbert[v]\right]
\end{align*}
into this equation we find, after some algebra, that $v = \Psi + \Psi^*$ satisfies \eqref{cub2}.

When \eqref{psieqn} holds, a solution of \eqref{ve3'21} for $\Psi_{21}$ is given by
\begin{align*}
\Psi_{21}
& = \frac{1}{2} \QQ\bigg[(- \rho^2 + \sigma) |\Psi|^2 \Psi_x + i (\sigma + \rho^2) \Psi \hilbert[|\Psi|^2]_x + \sigma \Psi^2 \Psi^*_x\bigg]_x.
\end{align*}

In conclusion, given a solution $\Psi$ of \eqref{psieqn}, or equivalently $v$ of \eqref{cub2}, we have constructed a function $\ve V$ of the form \eqref{V} that satisfies \eqref{cde'} up to a residual of the order $\ve^4$.

\subsection{Residual estimates}
In this subsection, we obtain estimates for the residual of the approximate solution constructed above.
We observe that at each stage in the expansion of $V$ we increase the degree in $\Psi$ by one and introduce one additional $x$-derivative, so $V_k$ is of degree $k+1$ in $\Psi$ and involves $k$ derivatives with respect to $x$. Thus, in order to construct the approximate solution $\ve V \in C([0,T], H^n)$, which involves two derivatives of $\Psi$, we require $n_v \ge n+2$
derivatives in the solution $v$ of \eqref{cub2}.
In the next lemma, we use two more derivatives of $v$ to estimate the residual of $\ve V$, but we do not attempt to make the estimate sharp.

\begin{lemma}
\label{lem:bhres}
Let $n \geq 0$ be an integer and suppose that $v \in C([0, T], H^{n_v}(\R))$ with $n_v \ge n + 4$
is a solution of \eqref{cub2}. Then there exists a constant $C > 0$, depending on $\|v\|_{C([0, T], H^{n_v}(\R))}$, such that for all sufficiently small $\ve >0$ there is a function $\ve V \in C([0, T], H^{n_v - 2}(\R))$ of the form \eqref{V} whose residual \eqref{res} satisfies the estimate
\begin{equation}
\sup_{t \in [0, T / \ve^2]} \|\Res(\ve V)(\cdot, t; \ve)\|_{H^n} \leq C \ve^4.
\label{resest}
\end{equation}
Furthermore,
\begin{align}
\label{diff'}
\sup_{t \in [ 0, T / \ve^2]} \left\|\ve \left[v(\cdot, \ve^2 t) \cos{t} + \hilbert[v](\cdot, \ve^2 t) \sin{t}\right] - \ve V(\cdot, t; \ve)\right\|_{H^n} \leq C \ve^2.
\end{align}
\end{lemma}

\begin{proof}
We constructed $\ve V \in C([0, T], H^{n_v - 2}(\R))$ in Section~\ref{sec:approx} in terms of $\Psi=\P[v]$. Using \eqref{V} and \eqref{N5-1} in \eqref{res}, together with the cancellations in \eqref{ve1'}--\eqref{ve3'}, we compute that its residual is given by
\begin{align*}
\Res(\ve V) & =  - \ve^4 \bigg[V_{1 \tau} + \frac{\rho}{2} (V_1^2)_x + \rho (V_0 V_2)_x\bigg] - \ve^5 \Big[V_{2 \tau} + \rho (V_1 V_2)_x\Big] - \ve^6 \Big[\rho V_2 V_{2 x}\Big]\\
& \qquad - \frac{\sigma}{2} \sum_{p = 1}^6 \sum_{\substack{0 \leq j ,k ,\ell \leq 2\\ j + k + \ell = p}} \ve^{p + 3} \px \left\{V_j V_k |\px| V_\ell - V_j |\px| (V_k V_\ell) + \frac{1}{3} |\px| (V_j V_k V_\ell)\right\} - \sigma \nonquin[\ve V].
\end{align*}
From  the expressions for $V_k$ with $0 \leq k \leq 2$, we see that there are at most four $x$-derivatives on $\Psi$ in all of the terms that involve the $V_k$, so their  $H^n$-norm can be estimated in terms of $\|v\|_{H^{n+4}}$. Similarly, we can use  Lemma~\ref{lem:noncub} to estimate the $H^n$-norm of $\nonquin[\ve V] = O(\ve^5)$, which gives \eqref{resest}.

The second inequality \eqref{diff'} follows directly from the construction of $V$.
\end{proof}

\section{A modified energy for the error}
\label{sec:bherr}

Given $\vp_0\in H^n$, let $\vp$ denote the solution of \eqref{cde'} with initial data $\vp(\cdot,0) = \ve \vp_0$, and let $\ve V$ denote the approximate solution \eqref{V} constructed from an asymptotic solution $v$ with the properties stated in Theorem~\ref{thm:euler-cub}.
We define a scaled error $R$ between the full and approximate solutions by
\begin{equation}
\ve^\beta R = \vp - \ve V.
\label{defR}
\end{equation}
In the following, we choose $\beta = 2$ and estimate $\|R\|_{H^n}$ for all sufficiently small $\ve > 0$, but we continue to denote the exponent by $\beta$ in order to make it easier to keep track of the error terms.
We note that \eqref{vests} and \eqref{diff'} ensure that $\|R\|_{H^n} = O(1)$ as $\ve\to 0^+$ at $t=0$.

Using $\vp = \ve^\beta R + \ve V$ in \eqref{cde2}, we obtain that
\begin{align}
\label{err}
R_t + \ve^\beta \rho R R_x + \ve \rho (V R)_x + \ve^{-\beta} \sigma (\noncub[\ve^\beta R + \ve V] - \noncub[\ve V] )  = \hilbert[R] + \ve^{-\beta} \Res(\ve V),
\end{align}
where $\Res(\ve V)$ is defined by \eqref{res}.

We will see that the term $\ve^{-\beta} \sigma \big(\noncub[\ve^\beta R + \ve V] - \noncub[\ve V]\big)$ is of the order $\ve^2$. The most dangerous terms in \eqref{err} are $\ve(V R)_x$ and $\ve^\beta RR_x$. They can be removed by a normal form transformation
$R\mapsto \bar{R}$ where
\begin{align}
\label{nf}
\bar{R} = R + \ve \rho \hilbert\left[\hilbert[V] \hilbert[R]\right]_x + \frac{1}{2} \ve^\beta\rho \hilbert\left[(\hilbert[R])^2\right]_x,
\end{align}
which yields a cubically nonlinear equation for $\bar{R}$. However, this equation contains second-order spatial derivatives in the nonlinearity, resulting in a loss of derivatives in its energy estimates, and the straightforward normal form transformation \eqref{nf} is not effective.

Following \cite{CW17,Dul17, HITW15}, we instead use \eqref{nf} to define a modified energy that is obtained by neglecting the
higher-order terms with the most derivatives from $\int |\partial^n \bar{R}|^2 \diff x$, where $n\ge 0$ is an integer. This procedure gives
\begin{align}
\label{eng}
E_n = \int |\partial^n R|^2 \diff x + 2 \ve \rho \int \partial^{n + 1} \hilbert\left[\hilbert[V] \hilbert[R]\right] \partial^n R \diff x + \ve^\beta \rho \int \partial^{n + 1} \hilbert\left[(\hilbert[R])^2\right] \partial^n R \diff x.
\end{align}
The first term in \eqref{eng} is the standard $\dot{H}^n$-energy of $R$, the second term cancels the leading order effect of $\ve (V R)_x$ on the time evolution of this norm, and the third term cancels the effect of $\ve^\beta RR_x$. For $n\in \N$, we then define the nonhomogeneous energy
\begin{align}
\label{toteng}
E = E_0 + E_n.
\end{align}
As stated in the next lemma, this energy is equivalent to the $H^n$-energy of $R$ for sufficiently small $\ve$.
\begin{lemma}
\label{lem:En}
Let $n \geq 2$ be an integer and $M >0$ any positive real number. Define $E$ by  \eqref{eng}--\eqref{toteng}, where
\begin{equation}
\|V\|_{H^{n + 2}} + \|R\|_{H^2} \le M.
\label{VRest}
\end{equation}
Then there exists a constant $c = c(n,M) > 0$ such that
\[
\frac{1}{c} \|R\|_{H^n}^2 \leq E \leq c \|R\|_{H^n}^2
\]
 for all sufficiently small $\ve > 0$.
\end{lemma}

\begin{proof}
For the first and third terms in \eqref{eng}, as in the proof of Lemma~2 in \cite{HITW15}, we have
\[
\int |\partial^n R|^2 \diff x + \ve^\beta \rho \int \partial^{n + 1} \hilbert\left[(\hilbert[R])^2\right] \partial^n R \diff x = \left[1 + O(\ve^\beta \|\hilbert[R_x]\|_{L^{\infty}})\right] \|\partial^n R\|_{L^2}^2 \qquad \text{as}\ \ve \to 0.
\]
Using the skew-adjointness of $\hilbert$, integration by parts, and H\"{o}lder's inequality, we can estimate the second term in \eqref{eng} by
\begin{align*}
& \left| 2 \ve \rho\int \partial^{n + 1} \hilbert\left[\hilbert[V] \hilbert[R]\right] \partial^n R \diff{x}\right|\\
= ~& \left|\ve \rho \int \hilbert[V_x] (\partial^n \hilbert[R])^2 \diff{x} -  2 \ve \rho\sum_{j = 0}^n \begin{pmatrix}n + 1\\ j\end{pmatrix} \partial^{n + 1 - j} \int \hilbert[V] \partial^j \hilbert[R] \partial^n \hilbert[R] \diff{x}\right|
\\
\leq ~& C \ve \left(\|\hilbert[ V_x]\|_{W^{n, \infty}} \|\partial^n R\|_{L^2}^2 + \|\hilbert[ V_x]\|_{W^{n, \infty}}\|R\|_{L^2}^2 \right) .
\end{align*}
The same estimates with $n=0$ hold for $E_0$, which controls $\|R\|_{L^2}$.
By \eqref{VRest} and Sobolev embedding,
\[
\|\hilbert[ V_x]\|_{W^{n, \infty}} +  \|\hilbert[R_x]\|_{L^{\infty}}\le M,
\]
so the equivalence of $E$ with $\|R\|_{H^n}$ follows for sufficiently small $\epsilon >0$.
\end{proof}

In the following, we fix an integer $n \ge 3$ in the energy \eqref {toteng} for equation \eqref{cde'} with $\sigma\ne 0$,
or $n\ge 2$ for the Burgers-Hilbert equation with $\sigma=0$.

\section{Modified energy estimates}
\label{sec:bhest}

In the rest of the paper, we prove that there exists a constant $C$, independent of $\ve$, such that the energy $E = E_0 + E_n$ defined in \eqref{eng}--\eqref{toteng} satisfies the estimate
\begin{align}
\label{engest}
E(t) \leq C \qquad \text{for all $0 \le t \le T/\ve^{2}$}
\end{align}
for all sufficiently small $\ve > 0$.
Then, by the equivalence of $E$ with the $H^n$-energy of $R$, we have
$\|R(\cdot,t)\|_{H^n} \le C$, and from definition of $R$ in \eqref{defR}, we obtain that
 \begin{equation}
 \|\vp(\cdot,t) - \ve V(\cdot,t;\ve)\|_{H^n} \le C\ve^\beta\qquad \text{for all $0 \le t \le T/\ve^{2}$},
 \label{engest1}
 \end{equation}
 where $\beta=2$.
Combining this result with \eqref{diff'} in Lemma~\ref{lem:bhres}, we obtain Theorem~\ref{thm:euler-cub}.

The main part of the proof is an \emph{a priori} estimate for $E$ when $R$ is a sufficiently smooth solution of \eqref{err}. It suffices to consider the evolution of $E_n$, since $E_0$ can be shown to satisfy the same estimate with $n$ replaced by $0$; in fact, the estimate for $E_0$ is easier.

In proving \eqref{engest}, we will use the following commutator estimate whose proof can be found in \cite{DMP08}.
\begin{lemma}
\label{lem:comm}
Let $\hilbert$ denote the Hilbert transform. Then for any $p \in (1, \infty)$, $\ell_1, \ell_2 \in \N$, $f \in L^p$, and
$a \in W^{{\ell_1 + \ell_2},\infty}$,
there exists $C = C(p, \ell_1, \ell_2) > 0$ such that
\[
\|\partial^{\ell_1} [\hilbert, a] \partial^{\ell_2} f\|_{L^p} \leq C \|\partial^{\ell_1 + \ell_2} a\|_{L^\infty} \|f\|_{L^p}.
\]
\end{lemma}

Time differentiating \eqref{eng}, we obtain that
\begin{align*}
\frac{1}{2} \frac{\diff}{\diff{t}}E_n = ~&\int \partial^n R \partial^n R_t \diff x + \ve \rho \int \partial^{n+1} \hilbert\left[\hilbert[V_t] \hilbert[R]\right] \partial^n R \diff x \\
&+ \ve \rho \int \partial^{n+1} \hilbert\left[\hilbert[V] \hilbert[R_t]\right] \partial^n R \diff x + \ve \rho \int \partial^{n+1} \hilbert\left[\hilbert[V] \hilbert[R]\right] \partial^n R_t \diff x\\
&+ \ve^\beta \rho \int \partial^{n + 1} \hilbert\left[\hilbert[R] \hilbert[R_t]\right] \partial^n R \diff{x} + \frac{\ve^\beta}{2} \rho \int \partial^{n+1} \hilbert\left[(\hilbert[R])^2\right] \partial^n R_t \diff x.
\end{align*}
Using \eqref{res} to eliminate $\ve V_t$ in terms of $\Res(\ve V)$ and \eqref{err} and to eliminate $R_t$, we get that
\begin{align}
\label{dtE}
\begin{split}
\frac{1}{2} \frac{\diff}{\diff{t}}E_n = ~&\int \partial^n R \partial^n \hilbert[R] \diff x  - \ve \rho \int \partial^n R \partial^{n + 1} (V R) \diff x +  \ve \rho \int \partial^{n + 1} \hilbert[ \hilbert^2[V] \hilbert[R]] \partial^n R \diff x \\
& + \ve \rho \int \partial^{n + 1} \hilbert[ \hilbert[V] \hilbert^2[R]] \partial^n R \diff x+ \ve \rho \int \partial^{n + 1} \hilbert[ \hilbert[V] \hilbert[R]] \partial^n \hilbert[R] \diff x\\
&- \ve^\beta \rho \int \partial^n R \partial^n (R R_x) \diff{x} + \ve^\beta \rho \int \partial^{n + 1} \hilbert\left[\hilbert[R] R\right] \partial^n R \diff{x}\\
&+ \frac{\ve^\beta}{2} \rho \int \partial^{n + 1} \hilbert\left[(\hilbert[R])^2\right] \partial^n \hilbert[R] \diff{x} + \rho^2 (I_1 + I_2 + I_3 + I_4 + I_5 + I_6 + I_7 + I_8 + I_9)\\
& + \rho (I_{11} + I_{12} + I_{13}) + I_{10} - \sigma (J_1 + J_2 + J_3 + J_4 + J_5 + J_6),
\end{split}
\end{align}
where, as we will show, the leading order terms cancel, and $I_\ell$ ($\ell = 1, \dotsc, 13$), $J_k$ ($k = 1, \dotsc, 6$) are terms of order
$\ve^2$ or higher. The terms $I_\ell$ are given explicitly by
\begin{alignat*}{2}
I_1 &=  - \ve^2 \int \partial^{n + 1} \hilbert[\hilbert[V V_x] \hilbert[R]] \partial^n R \diff{x}, &&\hspace{.1in}
I_2 = - \ve^2 \int \partial^{n + 1} \hilbert[\hilbert[V] \partial \hilbert[V R]] \partial^n R \diff{x},\\
I_3 &=  -\ve^2 \int \partial^{n + 1} \hilbert[\hilbert[V] \hilbert[R]] \partial^{n + 1} (V R) \diff{x}, &&\hspace{.1in}
I_4 =  - \ve^{\beta + 1} \int \partial^{n + 1} \hilbert[\hilbert[V] \hilbert[R]] \partial^n (R R_x) \diff{x},\\
I_5 &= - \ve^{\beta + 1} \int \partial^{n + 1} \hilbert[\hilbert[V] \hilbert[R R_x]] \partial^n R \diff{x}, &&\hspace{.1in}
I_6 = - \ve^{\beta + 1} \int \partial^{n + 1} \hilbert\left[\hilbert[R] \hilbert[VR]_x\right] \partial^n R \diff{x},\\
I_7 & = - \frac{\ve^{\beta + 1}}{2} \int \partial^{n + 1} \hilbert\left[(\hilbert[R])^2\right] \partial^{n + 1}(VR) \diff{x}, &&\hspace{.1in}
I_8 =- \ve^{2 \beta} \int \partial^{n + 1} \hilbert\left[\hilbert[R] \hilbert[R R_x]\right] \partial^n R \diff{x},\\
I_9 & = - \frac{\ve^{2 \beta}}{2} \int \partial^{n + 1} \hilbert\left[(\hilbert[R])^2\right] \partial^n (R R_x) \diff{x},&&\hspace{.1in}
I_{10} = \ve^{- \beta} \int \partial^n R \partial^n (\Res(\ve V)) \diff{x},\\
I_{11} & = \ve^{- \beta + 1} \int \partial^{n + 1} \hilbert[\hilbert[V] \hilbert[\Res(\ve V)]] \partial^n R \diff{x},  &&\hspace{.1in}
I_{12} = \ve^{- \beta + 1} \int \partial^{n+1} \hilbert[ \hilbert[V] \hilbert[R]] \partial^n \Res(\ve V) \diff x,\\
I_{13} & = \frac{1}{2} \int \partial^{n + 1} \hilbert\left[(\hilbert[R])^2\right] \partial^n \Res(\ve V) \diff{x},
\end{alignat*}
while the terms $J_k$ are given by
\begin{align*}
J_1 &= \ve^{-\beta}\int \partial^n R \partial^n \left(\noncub[\ve^\beta R + \ve V] - \noncub[\ve V]\right) \diff{x},
\\
 J_2 &= \ve \rho \int \partial^{n+1} \hilbert[ \hilbert[\noncub[\ve V]] \hilbert[R]] \partial^n R \diff x,\\
J_3 &= \ve^{-\beta + 1} \rho \int \partial^{n+1} \hilbert\left[\hilbert[V] \hilbert\left[\noncub[\ve^\beta R + \ve V] - \noncub[\ve V]\right]\right] \partial^n R \diff x,\\
J_4 &= \ve^{-\beta + 1} \rho \int \partial^{n+1} \hilbert[ \hilbert[V] \hilbert[R]] \partial^n \left(\noncub[\ve^\beta R + \ve V] - \noncub[\ve V]\right) \diff x,\\
J_5 &= \rho \int \partial^{n + 1} \hilbert\left[\hilbert[R] \hilbert\left[\noncub[\ve^\beta R + \ve V] - \noncub[\ve V]\right]\right] \partial^n R \diff{x},\\
J_6 & = \frac{\rho}{2} \int \partial^{n + 1} \hilbert\left[(\hilbert[R])^2\right] \partial^n \left(\noncub[\ve^\beta R + \ve V] - \noncub[\ve V]\right) \diff{x}.
\end{align*}

The first term on the right-hand side of  \eqref{dtE} vanishes due to the skew-adjointedness of the Hilbert transform. Making use of  the skew-adjointness of $\hilbert$, the fact that $\hilbert^2 = -\I$, and the Cotlar identity
\[
\hilbert\left[a b - \hilbert[a] \hilbert[b]\right] = a \hilbert[b] + b \hilbert[a],
\]
we find (as a consequence of the choice of the modified energy) that the terms of the order $\ve$ on the right-hand side of in \eqref{dtE} also vanish:
\begin{align*}
&- \int \partial^n R \partial^{n + 1} (V R) \diff x + \int \partial^{n + 1} \hilbert[ \hilbert^2[V] \hilbert[R]] \partial^n R \diff x \\
& \quad + \int \partial^{n + 1} \hilbert[ \hilbert[V] \hilbert^2[R]] \partial^n R \diff x+\int \partial^{n + 1} \hilbert[ \hilbert[V] \hilbert[R]] \partial^n \hilbert[R] \diff x\\
= ~& \int \partial^n R \partial^{n+1}[ - V R - \hilbert[V \hilbert[R]] - \hilbert[R \hilbert[V]] + \hilbert[V]\hilbert[R] ] \diff{x} = 0.
\end{align*}
Similarly, we have
\begin{align*}
&- \ve^\beta \rho \int \partial^n R \partial^n (R R_x) \diff{x} + \ve^\beta \rho \int \partial^{n + 1} \hilbert\left[\hilbert[R] R\right] \partial^n R \diff{x} + \frac{\ve^\beta}{2} \rho \int \partial^{n + 1} \hilbert\left[(\hilbert[R])^2\right] \partial^n \hilbert[R] \diff{x}\\
= ~& \frac{\ve^\beta}{2} \rho \int \partial^{n + 1} R \partial^n \left[R^2 - 2 \hilbert\left[\hilbert[R] R\right] - (\hilbert[R])^2\right] \diff{x} = 0.
\end{align*}
Thus, \eqref{dtE} reduces to
\begin{align}
\label{dtE1}
\begin{split}
\frac{1}{2} \frac{\diff}{\diff{t}}E_n = ~&
 \rho^2 (I_1 + I_2 + I_3 + I_4 + I_5 + I_6 + I_7 + I_8 + I_9)\\
& + \rho (I_{11} + I_{12} + I_{13}) + I_{10} - \sigma (J_1 + J_2 + J_3 + J_4 + J_5 + J_6),
\end{split}
\end{align}
and it suffices to estimate the error terms $I_\ell$ ($\ell = 1, \dotsc, 13$) and $J_k$ ($k = 1, \dotsc, 6$).

In order to do this, we make a bootstrap assumption
\begin{align}
\label{btsp}
E^{1/2} \leq \frac{1}{\varepsilon} \qquad \text{for all $0\le t \le T/\ve^{2}$}.
\end{align}
We then show in the following subsections that all of the error terms can be estimated by $\ve^2$ or $\ve^2 E$, which closes the bootstrap and establishes \eqref{engest}.

\subsection{Quadratic terms of order $\ve^2$}\label{sec:bhhot1}
These terms are $I_1$--$I_3$.
In this subsection, we use $\mathscr{R}_1$ to denote terms that might change from line to line and satisfy the estimate
\[
|\Rs_1| \leq C \ve^2 E.
\]
Error terms that contain at most $n$ derivatives on $R$ satisfy this estimate by a combination of Sobolev embedding and the Cauchy-Schwarz inequality, while terms that contain strictly fewer than $n$ derivatives on one $R$ and $n+1$ derivatives on another $R$ can be converted into terms of the previous type by an integration by parts. Thus, we only need to consider terms with $n+1$ derivatives on two factors of $R$ and terms with $n$ derivatives on one $R$ and $n+1$ derivatives on another $R$.

After applying the Leibniz rule, the worst term in $I_1$ can be handled by integration by parts as follows:
\begin{align*}
I_1 &= -\ve^2\int \hilbert[V V_x] \partial^n \hilbert[R] \partial^{n+1} \hilbert[R] \diff{x} + \Rs_1\\
&= \frac{1}{2}\ve^2 \int \partial \hilbert[V V_x]( \partial^n \hilbert[R] )^2 \diff{x} + \Rs_1\\
& \leq C \ve^2 E.
\end{align*}

The terms $I_2$ and $I_3$ permit a cancellation when all the derivatives hit $R$ or $\hilbert[R]$. Integrating by parts and using the skew-adjointness of $\hilbert$, we obtain that
\begin{align*}
I_2 + I_3 &= -\ve^2\int \hilbert[V] \hilbert[V \partial^{n+1} R] ] \partial^{n+1} \hilbert[R] \diff x - (n + 1) \ve^2\int \hilbert[V] \hilbert[ \partial V \partial^n R] \partial^{n+1} \hilbert[R] \diff x\\
& \quad -n \ve^2\int \hilbert[\partial V] \hilbert[V \partial^n R] \partial^{n+1} \hilbert[R] \diff x + \ve^2\int \hilbert[V] \partial^{n+1} \hilbert[R] \hilbert[V \partial^{n+1} R] \diff x \\
& \quad + (n+1)\ve^2\int \hilbert[\partial V] \partial^n \hilbert[R] \hilbert[V \partial^{n+1} R] \diff x + (n+1) \ve^2\int \hilbert[V] \partial^{n+1} \hilbert[R] \hilbert[ \partial V \partial^n R] \diff x+ \Rs_1\\
&=  \ve^2 \int \hilbert[\partial V] \hilbert[V \partial^{n + 1} R] \partial^n \hilbert[R] \diff{x} + \Rs_1.
\end{align*}
For the first term on the right-hand side, we make use of Lemma~\ref{lem:comm} to get
\begin{align*}
& \ve^2\int \hilbert[\partial V] \hilbert[V \partial^{n + 1} R] \partial^n \hilbert[R] \diff x\\
&\qquad = \ve^2\int \hilbert[\partial V] V \partial^{n+1} \hilbert[R] \partial^n \hilbert[R] \diff x + \ve^2\int \hilbert[\partial V] \left([\hilbert, V] \partial^{n+1} R\right) \partial^n \hilbert[R] \diff x \\
&\qquad \leq C\ve^2 \Big(\|\partial^n R\|_{L^2}^2 + \left\| [\hilbert, V] \partial^{n+1} R \right\|_{L^2} \|R\|_{L^2}\Big) \\
&\qquad \leq C \ve^2 E.
\end{align*}

\subsection{Cubic terms of order $\ve^{\beta + 1}$}
 We now consider the terms $I_4$--$I_7$ of order $\ve^{\beta + 1}$. Since these terms are cubic in $R$, we bound them by $E^{3/2}$ and use the bootstrap assumption \eqref{btsp}. We denote by $\mathscr{R}_2$ terms that satisfy the estimate
 \[
 |\mathscr{R}_2| \leq C\ve^{\beta + 1} E^{3/2}.
 \]
 As before, terms with at most $n$ derivatives on $R$ satisfy this estimate, and the only terms that we cannot reduce to this case using integration by parts are ones that either contain $n+1$ derivatives on two $R$ factors or $n+1$ derivatives on one $R$ factor and $n$ derivatives on another $R$ factor.

We first estimate $I_4$ and $I_5$. By the Leibniz rule and the skew-adjointness of $\hilbert$, we have
\begin{align*}
I_4 &= \ve^{\beta + 1}\int \hilbert[V] \partial^{n+1}\hilbert[R] \hilbert[R \partial^{n+1}R] \diff x + n \ve^{\beta + 1}\int \hilbert[V] \partial^{n+1} \hilbert[R] \hilbert[\partial R \partial^n R] \diff x \\
~&\quad + (n+1) \ve^{4} \int \hilbert[\partial V] \partial^n \hilbert[R] \hilbert[R \partial^{n+1} R] \diff x + \mathscr{R}_2,\\
I_5 &= -\ve^{\beta + 1} \int \hilbert[V] \hilbert[ R \partial^{n+1} R ] \partial^{n+1} \hilbert[R] \diff x - (n+1) \ve^{\beta + 1} \int \hilbert[V] \hilbert[ \partial R \partial^n R] \partial^{n+1} \hilbert[R] \diff x + \mathscr{R}_2,
\end{align*}
so
\begin{align}
\label{I5+I6}
I_4 + I_5 = - \ve^{\beta + 1} \int \hilbert[V] \hilbert[ \partial R \partial^n R] \partial^{n+1} \hilbert[R] \diff x + (n+1) \ve^{\beta + 1} \int \hilbert[\partial V] \partial^n \hilbert[R] \hilbert[R \partial^{n+1} R] \diff x + \mathscr{R}_2.
\end{align}
Using integration by parts and the commutator estimate in Lemma~\ref{lem:comm}, we estimate the first term on the right-hand side of \eqref{I5+I6} by
\begin{align*}
& - \ve^{\beta + 1}\int \hilbert[V] \hilbert[ \partial R \partial^n R] \partial^{n+1} \hilbert[R] \diff x\\
=~& \frac{\ve^{\beta + 1}}{2} \int \hilbert[V]  \partial R \partial | \partial^n \hilbert[R]|^2 \diff x +\ve^{\beta + 1} \int \hilbert[V] \left([\hilbert, \partial R] \partial^{n+1} R\right) \partial^n \hilbert[R] \diff x + \mathscr{R}_2\\
=~& - \frac{\ve^{\beta + 1}}{2} \int \hilbert[V]  \partial^2 R | \partial^n \hilbert[R]|^2 \diff x + \ve^{\beta + 1}\int \hilbert[V] \left([\hilbert, \partial R] \partial^{n+1} R\right) \partial^n \hilbert[R] \diff x + \mathscr{R}_2\\
\leq~& C\ve^{\beta + 1} E^{3/2}.
\end{align*}
The second term in the right-hand side of \eqref{I5+I6} can be estimated similarly by
\begin{align*}
& - \ve^{\beta + 1}\int \hilbert[\partial V] \partial^n \hilbert[R] \hilbert[R \partial^{n+1} R] \diff x\\
=~& \frac{\ve^{\beta + 1}}{2} \int \partial \left(\hilbert[\partial V] R\right) |\partial^n \hilbert[R]|^2 \diff{x} - \ve^{\beta + 1}\int \hilbert[\partial V] \partial^n \hilbert[R] \left([\hilbert, R] \partial^{n+1} R\right) \diff x\\
\leq~& C \ve^{\beta + 1} E^{3/2}.
\end{align*}

The estimates for $I_6$ and $I_7$ are similar. Observe that
\begin{align*}
I_6 &= - \ve^{\beta + 1} \int \partial^n \hilbert[R] \partial \hilbert[V R] \partial^{n + 1} \hilbert[R] \diff{x} - \ve^{\beta + 1} \int \hilbert[R] \hilbert[V \partial^{n + 1} R] \partial^{n + 1} \hilbert[R] \diff{x}\\
& \quad - (n + 1) \ve^{\beta + 1} \int \hilbert[R] \hilbert[\partial V \partial^n R] \partial^{n + 1} \hilbert[R] \diff{x} - n \ve^{\beta + 1} \int \partial \hilbert[R] \hilbert[V \partial^n R] \partial^{n + 1} \hilbert[R] \diff{x} + \mathscr{R}_2,\\
I_7 & = \ve^{\beta + 1} \int \hilbert[R] \partial^{n + 1} \hilbert[R] \hilbert[V \partial^{n + 1}R] \diff{x} + (n + 1) \ve^{\beta + 1} \int \hilbert[R] \partial^{n + 1} \hilbert[R] \hilbert[\partial V \partial^n R] \diff{x}\\
& \quad + (n + 1) \ve^{\beta + 1} \int \partial \hilbert[R] \partial^n \hilbert[R] \hilbert[V \partial^{n + 1} R] \diff{x} + \mathscr{R}_2.
\end{align*}
We then estimate them together and use integration by parts and Lemma~\ref{lem:comm} to obtain
\begin{align*}
I_6 + I_7 & = - \ve^{\beta + 1} \int \partial^n \hilbert[R] \partial^{n + 1} \hilbert[R] \partial \hilbert[V R] \diff{x} - n \ve^{\beta + 1} \int \partial \hilbert[R] \partial^{n + 1} \hilbert[R] \hilbert[V \partial^n R] \diff{x}\\
& \quad + (n + 1) \ve^{\beta + 1} \int \partial \hilbert[R] \partial^n \hilbert[R] \hilbert[V \partial^{n + 1} R] \diff{x} + \mathscr{R}_2\\
& = - \ve^{\beta + 1} \int \partial^n \hilbert[R] \partial^{n + 1} \hilbert[R] \partial \left([\hilbert, V] [R]\right) \diff{x} - n \ve^{\beta + 1} \int \partial \hilbert[R] \partial^{n + 1} \hilbert[R] \left([\hilbert, V] [\partial^n R]\right) \diff{x}\\
& \quad + (n + 1) \ve^{\beta + 1} \int \partial \hilbert[R] \partial^n \hilbert[R] \left([\hilbert, V] [\partial^{n + 1} R]\right) \diff{x} + \mathscr{R}_2\\
& \leq C \ve^{\beta + 1} E^{3 / 2},
\end{align*}
where the last inequality follows from integration by parts and the commutator estimates in Lemma~\ref{lem:comm}.

Using the bootstrap assumption \eqref{btsp}, and the fact that $\beta = 2$, we then have
\[
I_4 + I_5 + I_6 + I_7 \leq C \ve^{\beta + 1} E^{3 / 2} \leq C\ve^{2} E.
\]

\subsection{Quartic terms of order $\ve^{2 \beta}$}
The only quartic terms of order $\ve^{2 \beta}$ are $I_8$ and $I_9$, which also need to be estimated together. Since these terms are quartic in $R$, we bound them by $E^2$ and use the bootstrap assumption \eqref{btsp}. Let $\mathscr{R}_3$ denote terms that satisfy the estimate
\[
|\mathscr{R}_3| \leq C \ve^{2 \beta} E^2.
\]
We first observe that
\begin{align*}
I_8
& = -\ve^{2 \beta} \int \partial^{n + 1} \hilbert[R] \partial^n \hilbert[R] \hilbert[R R_x] \diff{x} - \ve^{2 \beta} \int \hilbert[R] \partial^n \hilbert[R R_x] \partial^{n + 1} \hilbert[R] \diff{x}\\
& \quad - n \ve^{2 \beta} \int \partial \hilbert[R] \hilbert[R \partial^n R] \partial^{n + 1} \hilbert[R] \diff{x} + \mathscr{R}_3,\\
I_9 & = \ve^{2 \beta} \int \hilbert[R] \partial^{n + 1} \hilbert[R] \partial^n \hilbert[R R_x] \diff{x} + (n + 1) \ve^{2 \beta} \int \partial \hilbert[R] \partial^n \hilbert[R] \partial^n \hilbert[R R_x] \diff{x} + \mathscr{R}_3.
\end{align*}
After an integration by parts, the first term on the right-hand side of $I_8$ becomes
\[
-\ve^{2 \beta} \int \partial^{n + 1} \hilbert[R] \partial^n \hilbert[R] \hilbert[R R_x] \diff{x} = \frac{1}{2}\ve^{2 \beta} \int |\partial^n \hilbert[R]|^2 \partial \hilbert[R R_x] \diff{x},
\]
which can be absorbed into $\mathscr{R}_3$. Summing $I_8$ and $I_9$ and canceling the identical terms, we obtain
\begin{align*}
I_8 + I_9 & = -\ve^{2 \beta} \int \partial^{n + 1} \hilbert[R] \partial^n \hilbert[R] \hilbert[R R_x] \diff{x} - n \ve^{2 \beta} \int \partial \hilbert[R] \hilbert[R \partial^n R] \partial^{n + 1} \hilbert[R] \diff{x}\\
& \quad + (n + 1) \ve^{2 \beta} \int \partial \hilbert[R] \partial^n \hilbert[R] \partial^n \hilbert[R R_x] \diff{x} + \mathscr{R}_3\\
& = \ve^{2 \beta} \int |\partial^n \hilbert[R]|^2 \partial [H, R] R_x \diff{x} - n \ve^{2 \beta} \int \partial^{n + 1} \hilbert[R] [H, R] \partial^n R \partial \hilbert[R] \diff{x} + \mathscr{R}_3\\
& \leq C \ve^{2 \beta} E^2
\\
&\leq C\ve^2 E,
\end{align*}
where the second-to-last inequality follows from integration by parts, the commutator estimates Lemma~\ref{lem:comm}, and, in the case when $n=2$ for the Burgers-Hilbert equation, the following pointwise estimate for $\delta > 0$
\[
\|[\hilbert, R_x] R_x\|_{L^\infty} + \|[\hilbert, R] \partial^2 R\|_{L^\infty} \leq C \|R_x\|_{H^{\frac{1}{2} + \delta}}^2.
\]

\subsection{Terms involving the residual}
These terms are  $I_{10}$--$I_{13}$. We can directly use H\"{o}lder's inequality, Sobolev embeddings, and Lemma~\ref{lem:bhres} to obtain that
\begin{align*}
I_{10} &\leq \ve^{-\beta} E^{1/2} \| \Res(\ve V)\|_{H^n} \leq C \ve^{4 - \beta} E^{1 / 2},\\[1ex]
I_{11} &\leq \ve^{-\beta + 1} \|V\|_{H^{n+1}} \|\Res(\ve V)\|_{H^{n+1}} E^{1/2} \leq C \ve^{5 - \beta} E^{1 / 2},\\
I_{12} &= - \ve^{-2}\int \partial^n \hilbert[ \hilbert[V] \hilbert[R]] \partial^{n+1} \Res(\ve V) \diff x \\
&\leq \ve^{-\beta + 1} \|V\|_{H^n} \| \Res(\ve V) \|_{H^{n+1}}  E^{1/2} \leq C \ve^{5 - \beta} E^{1 / 2},\\
I_{13} & \leq C \|R\|_{H^n}^2 \|\Res(\ve V)\|_{H^{n + 1}} \leq C \ve^{5} E.
\end{align*}
Here, since we require the $H^{n+1}$-norm of the residual of the asymptotic solution $\ve V \in C([0,T/\ve^2], H^{n_v}$), we need to take $n_v \ge n+5$ in order to apply Lemma~\ref{lem:bhres}.

\subsection{Higher degree terms}

In this subsection, we estimate the terms $J_k$, $k = 1, \dotsc, 6$. These terms do not appear for the Burgers-Hilbert equation
with $\sigma = 0$.

We will use $\mathscr{R}$ to denote terms that involve lower order derivatives and satisfy a straightforward estimate
\[
|\mathscr{R}| \leq C\ve^{2} E.
\]
We also use the notation
\begin{equation}
\Delta_\zeta f(x) = f(x) - f(x+\zeta) \quad \text{and} \quad D_\zeta f(x) = \frac{\Delta_\zeta f(x)}{\zeta}
\label{def_Delta}
\end{equation}
to denote differences and difference quotients, where we show the dependence on  the spatial variables explicitly but suppress the time variable.

\subsubsection{Sobolev energy term  $J_1$}

Using \eqref{noncub} in the expression for $J_1$ and writing the result in terms of the notation in \eqref{def_Delta}, we get that
\begin{align*}
J_1
& = \frac{1}{2 \pi} \ve^{-\beta}\int \partial^n R(x) \partial^n \int \left[ \ve^\beta R_x(x) - \ve^\beta R_x(x+ \zeta) \right] \log\bigg[1 + \frac{[\ve^\beta \Delta_\zeta R(x) + \ve \Delta_\zeta V(x)]^2}{|\zeta|^2}\bigg] \diff{\zeta} \diff{x} \\
&+\frac{1}{2 \pi} \ve^{-\beta}\int \partial^n R(x) \partial^n \int \left[  \ve \Delta_\zeta V(x) \right]\bigg\{ \log\bigg[1 + \frac{[\ve^\beta \Delta_\zeta R(x) + \ve \Delta_\zeta V(x)]^2}{|\zeta|^2}\bigg]- \log\bigg[1 + \frac{[ \ve \Delta_\zeta V(x)]^2}{|\zeta|^2}\bigg] \bigg\}\diff{\zeta} \diff{x} \\
& = \frac{1}{2 \pi} \left(J_{1,1} + J_{1,2} + J_{1,3} + J_{1,4} + J_{1,5}\right) + \mathscr{R},
\end{align*}
where
\begin{align*}
J_{1,1} &= \iint \partial^n R(x) \partial^n R_x(x) \log\bigg[1 + \frac{[\ve^\beta \Delta_\zeta R(x) + \ve \Delta_\zeta V(x)]^2}{|\zeta|^2}\bigg] \diff{\zeta} \diff{x}, \\
J_{1,2} &= - \iint \partial^n R(x) \partial^n R_x(x+\zeta) \log\bigg[1 + \frac{[\ve^\beta \Delta_\zeta R(x) + \ve \Delta_\zeta V(x)]^2}{|\zeta|^2}\bigg] \diff{\zeta} \diff{x},  \\
J_{1,3} &= \iint \partial^n R(x)  \Delta_\zeta R_x(x) \partial^n \log\left[1 + \frac{[\ve^\beta \Delta_\zeta R(x) + \ve \Delta_\zeta V(x)]^2}{|\zeta|^2}\right] \diff{\zeta} \diff{x},   \\
J_{1,4} &= \ve^{-\beta + 1}\iint \partial^n R(x) \partial^n  \Delta_\zeta V(x) \bigg\{ \log\bigg[1 + \frac{[\ve^\beta \Delta_\zeta R(x) + \ve \Delta_\zeta V(x)]^2}{|\zeta|^2}\bigg]- \log\bigg[1 + \frac{[ \ve \Delta_\zeta V(x)]^2}{|\zeta|^2}\bigg] \bigg\}\diff{\zeta} \diff{x}, \\
J_{1,5} &= \ve^{- \beta + 1}\iint \partial^n R(x) \Delta_\zeta V(x) \partial^n \bigg\{ \log\bigg[1 + \frac{[\ve^\beta \Delta_\zeta R(x) + \ve \Delta_\zeta V(x)]^2}{|\zeta|^2}\bigg]- \log\bigg[1 + \frac{[ \ve \Delta_\zeta V(x)]^2}{|\zeta|^2}\bigg] \bigg\}\diff{\zeta} \diff{x}.
\end{align*}
When $\partial^n$ hits $R_x(x)$, we can form a total derivative and integrate by parts
\begin{align*}\
J_{1,1}
&~ = -\iint\ \frac{ |\partial^n R(x)|^2}{2} \partial_x \log\bigg[1 + \frac{[\ve^\beta \Delta_\zeta R(x) + \ve \Delta_\zeta V(x)]^2}{|\zeta|^2}\bigg] \diff{\zeta} \diff{x}\\
&~ = -\iint|\partial^n R(x)|^2 \frac{ [\ve^\beta \Delta_\zeta R(x) + \ve \Delta_\zeta V(x)] [\ve^\beta \Delta_\zeta R_x(x) + \ve \Delta_\zeta V_x(x)] }{ \zeta^2 + [\ve^\beta \Delta_\zeta R(x) + \ve \Delta_\zeta V(x)]^2 } \diff{\zeta} \diff{x}\\
&~ \leq \iint |\partial^n R(x)|^2 \frac{ [\ve^\beta \Delta_\zeta R(x) + \ve \Delta_\zeta V(x)] }{ \zeta } \cdot \frac{ [\ve^\beta \Delta_\zeta R_x(x) + \ve \Delta_\zeta V_x(x)] }{ \zeta }\diff{\zeta} \diff{x}\\
&~ = J_{1,1,1} + J_{1,1,2},
\end{align*}
where
\begin{align*}
J_{1,1,1} &= \int_\R \int_{ |\zeta| > 1}  |\partial^n R(x)|^2\frac{ [\ve^\beta \Delta_\zeta R(x) + \ve \Delta_\zeta V(x)] [\ve^\beta \Delta_\zeta R_x(x) + \ve \Delta_\zeta V_x(x)] }{ \zeta^2 } \diff{\zeta} \diff{x},\\
J_{1,1,2} &= \int_\R \int_{|\zeta| < 1} |\partial^n R(x)|^2 \frac{1}{\zeta^{1/2}}\bigg[ \frac{ \ve^\beta \Delta_\zeta R(x) }{\zeta} + \frac{\ve \Delta_\zeta V(x) }{ \zeta } \bigg] \cdot \bigg[\frac{ \ve^\beta \Delta_\zeta R_x(x)}{|\zeta|^{1/2}} + \frac{\ve \Delta_\zeta V_x(x) }{ |\zeta|^{1/2} } \bigg]\diff{\zeta} \diff{x}.
\end{align*}

When $|\zeta|$ is large, we use the Sobolev embedding theorem and the fact that $\zeta \mapsto \zeta^{-2}$ is integrable at infinity
to conclude that
\[
J_{1,1,1} \leq C \|\partial^n R\|_{L^2}^2 ( \ve^\beta \|R\|_{L^\infty} + \ve \|V\|_{L^\infty})( \ve^\beta \|R_x\|_{L^\infty} + \ve \|V_x\|_{L^\infty}).
\]
When $|\zeta|$ is small, we use the fact that $\zeta \mapsto |\zeta|^{-1/2}$ is locally integrable, and distribute the remaining $|\zeta|^{3/2}$ in the denominator to form difference quotients and Holder norms. We then bound the difference quotients by Sobolev norms
to get
\[
J_{1,1,2} \leq C \|\partial^n R\|_{L^2}^2 (\ve^\beta \|R_x\|_{L^\infty} + \ve \|V_x\|_{L^\infty} )( \ve^\beta \|R_x\|_{C^{0,1/2}} + \ve \|V_x\|_{C^{0,1/2}}).
\]
It follows from the Sobolev embedding $H^1(\R) \hookrightarrow C^{0,1/2}(\R)$ that $J_{1,1}$ satisfies the estimate
\[
J_{1, 1} \leq C (\ve^2 E + \ve^{\beta + 1} E^{3 / 2} + \ve^{2 \beta} E^2).
\]

We now consider the term $J_{1,2}$ that arises when $\partial^n$ hits $R_x(x+\zeta)$, where we convert a derivative in $x$ to a derivative in $\zeta$ and integrate by parts. It follows that
\begin{align*}
J_{1, 2}
&~ = -\int \partial^n R(x) \int \partial^n R(x+\zeta) \partial_\zeta \log\bigg[1 + \frac{[\ve^\beta \Delta_\zeta R(x) + \ve \Delta_\zeta V(x)]^2}{|\zeta|^2}\bigg] \diff{\zeta} \diff{x} \\
&~ = -2\iint \partial^n R(x) \partial^n R(x+\zeta)
\\
&\qquad \bigg\{  [\ve^\beta \Delta_\zeta R(x) + \ve \Delta_\zeta V(x)] \frac{ \zeta[\ve^\beta R_\zeta(x+\zeta) + \ve V_\zeta(x + \zeta)] - [\ve^\beta \Delta_\zeta R(x) + \ve \Delta_\zeta V(x)] }{ [\zeta^2 + [\ve^\beta \Delta_\zeta R(x) + \ve \Delta_\zeta V(x)]^2] \zeta } \bigg\} \diff{\zeta}\diff{x}\\
&~ \leq C \left(J_{1,2,1} + J_{1,2,2}\right),
\end{align*}
where
\begin{align*}
J_{1,2,1} &= \int_\R \int_{|\zeta| > 1} | \partial^n R(x)| | \partial^n R(x + \zeta)|
\\
&\qquad \bigg\{[\ve^\beta \Delta_\zeta R(x) + \ve \Delta_\zeta V(x)][\ve^\beta R_\zeta(x+\zeta) + \ve V_\zeta(x + \zeta)] + [\ve^\beta \Delta_\zeta R(x) + \ve \Delta_\zeta V(x)]^2\biggr\} \frac{\diff{\zeta}}{\zeta^2} \diff{x},\\
J_{1,2,2} &= \int_\R \int_{ |\zeta| < 1} | \partial^n R(x)| | \partial^n R(x + \zeta)|
\\
&\qquad \bigg[ \frac{\ve^\beta \Delta_\zeta R(x) }{\zeta} + \frac{ \ve \Delta_\zeta V(x)}{\zeta} \bigg]  \bigg[ \ve^\beta \frac{ R_\zeta(x+\zeta) - D_\zeta R(x) }{|\zeta|^{1/2}} + \ve \frac{ V_\zeta(x+\zeta) - D_\zeta V}{|\zeta|^{1/2} } \bigg] \frac{\diff{\zeta}}{|\zeta|^{1/2}} \diff{x}.
\end{align*}
The $J_{1,2,1}$ integral over $|\zeta| > 1$ is treated as before, which gives
\[
J_{1,2,1} \leq C \|\partial^n R\|_{L^2}^2 ( \ve^\beta \|R\|_{L^\infty} + \ve \|V\|_{L^\infty})^2 + C \|\partial^n R\|_{L^2}^2 ( \ve^\beta \|R\|_{L^\infty} + \ve \|V\|_{L^\infty})( \ve^\beta \|R_x\|_{L^\infty} + \ve \|V_x\|_{L^\infty}).
\]

To treat $J_{1,2,2}$, where the integral is over $|\zeta| < 1$, we form a difference quotient and use the H\"{o}lder-norm bound
\[
\left\| \frac{ f_\zeta(x + \zeta) - D_\zeta f(x) }{ \zeta^{1/2}} \right\|_{ L_{x,\zeta}^\infty} \leq C \| f \|_{C^{0,1/2}},
\]
which follows from the usual definition of the H\"{o}lder norm and the mean value theorem applied to the difference quotient $D_\zeta f(x)$. We obtain the bound

\begin{align*}
J_{1,2} \leq C \|\partial^n R \|_{L^2}^2 ( \ve^\beta \|R\|_{W^{1, \infty}} + \ve \|V\|_{W^{1, \infty}} ) ( \ve^\beta \|R_\zeta\|_{C^{0,1/2}} + \ve \|V_{\zeta}\|_{C^{0,1/2}}) ).
\end{align*}

We next consider the term that arises when $\partial^n$ is applied to the logarithm, and in particular $\Delta_\zeta R$ through the chain rule. Letting
\[
g(x) = \ve^\beta R(x) + \ve V(x) \quad \text{and} \quad \Delta_\zeta g(x) = \ve^\beta \Delta_\zeta R(x) + \ve \Delta_\zeta V(x),
\]
we obtain
\begin{align*}
J_{1,3}
&= 2 \iint \partial^n R(x) \Delta_\zeta R_x(x) \partial^{n-1} \left[\left(|\zeta|^2 + |\Delta_\zeta g(x)|^2\right)^{-1} \Delta_\zeta g(x) \Delta_\zeta g_x(x) \right] \diff{\zeta} \diff{x}\\
&= 2 J_{1,3,1} + \left((-1)^{n - 1} 2^n (n - 1)!\right) J_{1,3,2} + \mathscr{R},
\end{align*}
where
\begin{align*}
J_{1,3,1} &= \iint \partial^n R(x) \Delta_\zeta R_x(x) \left(|\zeta|^2 + |\Delta_\zeta g(x)|^2\right)^{-1} \Delta_\zeta g(x) \partial^n \Delta_\zeta g (x)\diff{\zeta} \diff{x},\\
J_{1,3,2} &= \iint \partial^n R(x) \Delta_\zeta R_x(x) \left(|\zeta|^2 + |\Delta_\zeta g(x)|^2\right)^{-n} \left(\Delta_\zeta g(x) \Delta_\zeta g_x(x)\right)^n\diff{\zeta} \diff{x}.
\end{align*}

We again split the integrals into regions of large and small $\zeta$. For $J_{1,3,1}$, we have
\begin{align*}
J_{1,3,1} &\leq \int \int_{|\zeta| > 1} |\partial^n R(x)| |\Delta_\zeta R_x(x)| |\Delta_\zeta g(x)| |\partial^n \Delta_\zeta g(x)| \frac{1}{|\zeta|^2} \diff{\zeta} \diff{x}\\
& \qquad + \int \int_{|\zeta| < 1} |\partial^n R(x)| \frac{|\Delta_\zeta R_x(x)|}{|\zeta|^{1/2}} \frac{ |\Delta_\zeta g(x)| }{|\zeta|} |\partial^n \Delta_\zeta g(x)| \frac{1}{|\zeta|^{1/2}} \diff{\zeta} \diff{x}\\
& \leq C \|\partial^n R\|_{L^2} \|R_x\|_{C^{0, 1 / 2}} (\ve^\beta \|R\|_{W^{1, \infty}} + \ve \|V\|_{W^{1, \infty}}) (\ve^\beta \|\partial^n R\|_{L^2} + \ve \|\partial^n V\|_{L^2}).
\end{align*}

For the term $J_{1,3,2}$, we have
\begin{align*}
J_{1,3,2} & \leq \int \int_{|\zeta|>1} |\partial^n R(x) \Delta_\zeta R_x(x)| \left|\Delta_\zeta g(x) \Delta_\zeta g_x(x)\right|^n \frac{\diff{\zeta}}{\zeta^{2 n}} \diff{x},\\
&\qquad + \int \int_{|\zeta|<1} |\partial^n R(x)| | \Delta_\zeta R_x(x)|  \left( \frac{ | \Delta_\zeta g(x) | }{|\zeta|} \right)^n \left( \frac{ | \Delta_\zeta g_x(x) | }{|\zeta|} \right)^n \diff{\zeta} \diff{x}\\
& \leq C \|\partial^n R\|_{L^2} \|\partial R\|_{L^2} (\ve^\beta \|R\|_{W^{2, \infty}} + \ve \|V\|_{W^{2, \infty}} )^{2 n}.
\end{align*}

To handle $J_{1,4}$, we apply the mean value inequality to the difference of the logarithms to obtain
\begin{align*}
\bigg| \log\bigg[1 + \frac{[\ve^\beta \Delta_\zeta R(x) + \ve \Delta_\zeta V(x)]^2}{|\zeta|^2}\bigg]- \log\bigg[1 + \frac{[ \ve \Delta_\zeta V(x)]^2}{|\zeta|^2}\bigg] \bigg| \leq \frac{\diff}{\diff{c}}\bigg\vert_{c=c_*} \bigg[ \log\bigg( 1 + \frac{c^2}{\zeta^2} \bigg)\bigg] \cdot \ve^\beta \left| \Delta_\zeta R(x) \right|,
\end{align*}
where $c_*$ is a value between $c=\ve \Delta_\zeta V(x)$ and $c=\ve^\beta \Delta_\zeta R(x) + \ve \Delta_\zeta V(x)$ that maximizes
\[
c \mapsto \frac{\diff}{\diff{c}} \log\left[ 1 + \frac{c^2}{\zeta^2} \right] = \frac{2c}{ \zeta^2 + c^2}.
\]
Using $|c| \leq |\ve^\beta \Delta_\zeta R(x)| + |\ve \Delta_\zeta V(x)|$
we find that
\begin{align*}
J_{1,4} & \leq 2 \ve \iint |\partial^n R(x)|  | \Delta_\zeta \partial^n V(x)|  \frac{|\ve^{\beta} \Delta_\zeta R(x)| + \ve |\Delta_\zeta V(x)|}{\zeta^2}   | \Delta_\zeta R(x)| \diff{\zeta} \diff{x}\\
& \leq \ve^2 \|\partial^n R\|_{L^2} \|\partial^n V\|_{L^2} \|R\|_{W^{1, \infty}}\left(\ve^{\beta - 1} \|R\|_{W^{1, \infty}} + \|V\|_{W^{1, \infty}}\right),
\end{align*}
where the last inequality follows by splitting the integration regions as usual.

In the $J_{1,5}$ term, we start by taking one derivative of the difference of logarithms
\begin{align*}
\partial \bigg\{ \log\bigg[1 + \frac{\left(\Delta_\zeta g(x)\right)^2}{|\zeta|^2}\bigg]- \log\bigg[1 + \frac{[ \ve \Delta_\zeta V(x)]^2}{|\zeta|^2}\bigg] \bigg\} = \frac{ 2 \Delta_\zeta g(x) \Delta_\zeta g_x(x) }{ \zeta^2 + \left(\Delta_\zeta g(x)\right)^2} - \frac{ 2 \ve^2 \Delta_\zeta V(x) \Delta_\zeta V_x(x) }{ \zeta^2 + \left( \ve \Delta_\zeta V(x) \right)^2 }.
\end{align*}
When we consider only the terms that are quadratic in $V(x)$, we see that
\begin{align*}
&\frac{2 \ve^2 \Delta_\zeta V(x) \Delta_\zeta V_x(x)}{ \zeta^2 + \left(\Delta_\zeta g(x)\right)^2} - \frac{ 2 \ve^2 \Delta_\zeta V(x) \Delta_\zeta V_x(x)}{ \zeta^2 + \left( \ve \Delta_\zeta V(x) \right)^2 }
\\
&\qquad=2 \ve^2 \Delta_\zeta V(x) \Delta_\zeta V_x(x) \cdot \frac{  - \ve^{2 \beta} (\Delta_\zeta R(x) )^2 - 2 \ve^{\beta + 1} \Delta_\zeta R(x) \Delta_\zeta V(x)  }{\left[ \zeta^2 + ( \ve \Delta_\zeta V(x) )^2  \right]  \left[\zeta^2 + (\Delta_\zeta g(x))^2\right] }.
\end{align*}
Using these last two equalities in $J_{1,5}$, we have
\begin{align*}
J_{1,5} &= \ve^{\beta + 1} \iint \partial^n R(x) \Delta_\zeta V(x) \partial^{n-1}  \bigg\{ \frac{ 2 \Delta_\zeta R(x) \Delta_\zeta R_x(x) }{ \zeta^2 + (\Delta_\zeta g(x))^2} \bigg\} \diff{\zeta} \diff{x}\\
&~ \quad + \ve^{2}\iint \partial^n R(x) \Delta_\zeta V(x) \partial^{n-1} \bigg\{ \frac{ 2 \Delta_\zeta R(x) \Delta_\zeta V_x(x) }{ \zeta^2 + (\Delta_\zeta g(x))^2} \bigg\} \diff{\zeta} \diff{x}\\
&~ \quad + \ve^{2}\iint \partial^n R(x) \Delta_\zeta V(x) \partial^{n-1} \bigg\{ \frac{ 2 \Delta_\zeta V(x) \Delta_\zeta R_x(x) }{ \zeta^2 + (\Delta_\zeta g(x))^2} \bigg\} \diff{\zeta} \diff{x}\\
&~ \quad - \ve^{\beta + 3}\iint \partial^n R(x) \Delta_\zeta V(x) \partial^{n-1} \bigg\{ \frac{  2  \Delta_\zeta V(x) \Delta_\zeta V_x(x)  ( \Delta_\zeta R(x) )^2 }{\left[ \zeta^2 + ( \ve \Delta_\zeta V(x) )^2  \right]  \left[\zeta^2 + (\Delta_\zeta g(x))^2\right] }  \bigg\}\\
&~ \quad - \ve^{4}\iint \partial^n R(x) \Delta_\zeta V(x) \partial^{n-1} \bigg\{ \frac{ 4 \Delta_\zeta V(x) \Delta_\zeta V_x(x) \Delta_\zeta R(x) \Delta_\zeta V(x)  }{\left[ \zeta^2 + ( \ve \Delta_\zeta V(x) )^2  \right]  \left[\zeta^2 + (\Delta_\zeta g(x))^2\right] }  \bigg\}.
\end{align*}
Each of these terms can be handled similarly to $J_{1,3}$, and the resulting estimate is
\begin{align*}
J_{1, 5} & \leq C \ve^{\beta + 1} \|\partial^n R\|_{L^2}^2 \|R\|_{W^{1, \infty}} + C \ve^2 \|\partial^n R\|_{L^2} \|\partial^{n - 1} R\|_{L^2} + C \ve^2 \|\partial^n R\|_{L^2}^2\\
& \qquad + C \|\partial^n R\|_{L^2} (\ve^\beta \|R\|_{W^{2, \infty}} + \ve \|V\|_{W^{2, \infty}} )^2 + C \|\partial^n R\|_{L^2} (\ve^\beta \|R\|_{W^{2, \infty}} + \ve \|V\|_{W^{2, \infty}} )^{2 n}\\
& \qquad + C \|\partial^n R\|_{L^2} (\ve^\beta \|R\|_{W^{2, \infty}} + \ve \|V\|_{W^{2, \infty}} ) + C \|\partial^n R\|_{L^2} (\ve^\beta \|R\|_{W^{2, \infty}} + \ve \|V\|_{W^{2, \infty}} )^n\\
& \leq C \ve^2 E (1 + \ve^{\beta - 1} E^{1 / 2} + \ve^{2 \beta - 2} E^{1 / 2} + \ve^{2 \beta - 2} E + \ve^{\beta n  - 2} E^{(n - 1) / 2} + \ve^{2 \beta n  - 2} E^n).
\end{align*}

Putting these estimates together and making use of the bootstrap assumption \eqref{btsp}, we obtain
\[
J_1 \leq C \ve^2 E (1 + \ve^{\beta - 1} E^{1 / 2} + \ve^{2 \beta - 2} E^{1 / 2} + \ve^{2 \beta - 2} E + \ve^{\beta n  - 2} E^{(n - 1) / 2} + \ve^{2 \beta n  - 2} E^n) \leq C \ve^2 E.
\]

\subsubsection{Modified energy terms} The higher-order terms that involve the modified energy correction are $J_2$--$J_6$.
We begin with the term $J_2$. Using the skew-adjointness of the Hilbert transform and considering the term with the most derivatives
on $R$, we have
\begin{align*}
J_2 &= - \ve \rho \int \partial^{n+1} \left( \hilbert [\noncub[\ve V]]  \hilbert[R] \right) \partial^n \hilbert[R] \diff{x}\\
&= - \ve \rho \int  \hilbert [\noncub[\ve V]]  \partial^{n+1} \hilbert[R] \ \partial^n \hilbert[R] \diff{x} + \mathscr{R} \\
&=  - \frac{\ve}{2}\rho \int  \partial \noncub[\ve V]  \hilbert \left[\left( \partial^n \hilbert[R] \right)^2\right] \diff{x} + \mathscr{R}.
\end{align*}
We can obtain a pointwise bound for
\begin{align*}
\partial \noncub[\ve V](x) &= \frac{\ve}{2 \pi} \partial \int \Delta_\zeta V_x(x) \log\bigg[ 1 + \frac{ [\ve\Delta_\zeta V(x)]^2 }{\zeta^2 } \bigg] \diff{\zeta}\\
&= \frac{\ve}{2 \pi} \int \Delta_\zeta \partial V_x(x) \log\bigg[ 1 + \frac{ [\ve\Delta_\zeta V(x)]^2 }{\zeta^2 } \bigg] \diff{\zeta} + \frac{\ve}{2 \pi} \int \Delta_\zeta V_x(x) \frac{2 \ve^2  \Delta_\zeta V(x)\Delta_\zeta V_x (x)  }{ \zeta^2 + [\ve\Delta_\zeta V(x)]^2  } \diff{\zeta}\\
& \leq C \ve^3 \int_{|\zeta| \geq 1} \Delta_\zeta \partial V_x(x) [\Delta_\zeta V(x)]^2 \frac{\diff{\zeta}}{\zeta^2} + C \ve^3 \int_{|\zeta| < 1} \Delta_\zeta \partial V_x(x) [V_x(x)]^2 \diff{\zeta}\\
& \qquad + C \ve^3 \int_{|\zeta| \geq 1} \Delta_\zeta V_x(x) \Delta_\zeta V(x)\Delta_\zeta V_x (x) \frac{\diff{\zeta}}{\zeta^2}\\
& \leq C \ve^3,
\end{align*}
where the last inequality follows from the fact that $V \in H^{n_v - 2}(\R)$ and $C > 0$ depends on $V$.
This implies that
\begin{align*}
J_2 \leq C \ve^2 (1 + E).
\end{align*}

We next consider $J_3 + J_4$, which is given by
\begin{align*}
J_3+J_4
&= \ve^{1 - \beta} \rho \bigg[ \int \hilbert[V] \partial^n \hilbert\left[ \noncub[\ve^\beta R + \ve V] - \noncub[\ve V] \right] \partial^{n+1} \hilbert[R] \diff{x} \\
&~  \quad - \int  \hilbert[V] \partial^{n+1} \hilbert[R] \partial^n \hilbert\left[ \noncub[\ve^\beta R + \ve V] - \noncub[\ve V] \right]  \diff{x}  \\
&~ \quad - (n+1)  \int  \hilbert[\partial V] \partial^n \hilbert[R] \partial^n \hilbert\left[ \noncub[\ve^\beta R + \ve V] - \noncub[\ve V] \right]  \diff{x} \bigg] + \mathscr{R} \\
&= - (n + 1) \ve^{1 - \beta} \rho \int  \hilbert[\partial V] \partial^n \hilbert[R] \partial^n \hilbert\left[ \noncub[\ve^\beta R + \ve V] - \noncub[\ve V] \right]  \diff{x} + \mathscr{R}.
\end{align*}

The remaining term on the right-hand side of this equation can be estimated in a similar way to before. After using commutators to cancel the Hilbert transforms, we get that
\begin{align*}
&~ \ve^{1 - \beta} \int  \hilbert[\partial V] \partial^n \hilbert[R] \partial^n \hilbert\left[ \noncub[\ve^\beta R + \ve V] - \noncub[\ve V] \right]  \diff{x} \\
&=  \ve^{1 - \beta}  \bigg[\int  \hilbert[\partial V] \partial^n R \partial^n \left( \noncub[\ve^\beta R + \ve V] - \noncub[\ve V] \right)  \diff{x}\\
& \hspace{1in} - \int   \partial^n \hilbert[R] \left[\hilbert, \hilbert[\partial V]\right] \left[\partial^n \left( \noncub[\ve^\beta R + \ve V] - \noncub[\ve V] \right) \right]  \diff{x} \bigg].
\end{align*}
The first term is estimated in a similar way to $J_1$; the presence of the factor $\hilbert[\partial V]$ does not change the method. From Lemma~\ref{lem:comm}, the second commutator term satisfies the estimate
\begin{align*}
&~ \ve^{1 - \beta} \int   \partial^n \hilbert[R] \left[\hilbert, \hilbert[\partial V]\right] \left[\partial^n \left( \noncub[\ve^\beta R + \ve V] - \noncub[\ve V] \right) \right]  \diff{x}  \\
&\leq C \ve^{1 - \beta} \left\| \partial^n R\right\|_{L^2} \left\| \partial^{n + 1} \hilbert[V]\right\|_{L^\infty} \left\|\noncub[\ve^\beta R + \ve V] - \noncub[\ve V] \right\|_{L^2}.
\end{align*}
The important quantity to estimate here is the $L^2$ norm of the difference of the nonlinearities
\begin{align*}
&~ \left\| \noncub[\ve^\beta R + \ve V] - \noncub[\ve V] ] \right\|_{L^2}^2 \\
&= \frac{1}{4 \pi^2} \int \bigg\{ \int \ve^\beta \Delta_\zeta R_x(x)  \log \bigg( 1 + \frac{ (\Delta_\zeta g(x))^2}{\zeta^2} \bigg) \diff{\zeta}\\
& \hspace{1in} + \int \ve \Delta_\zeta V_x(x) \bigg[  \log \bigg( 1 + \frac{ (\Delta_\zeta g(x))^2}{\zeta^2} \bigg) \diff{\zeta} -  \log \bigg( 1 + \frac{ (\ve\Delta_\zeta V(x))^2}{\zeta^2} \bigg) \bigg] \diff{\zeta} \bigg\}^2 \diff{x}\\
&\leq C \int   \bigg\{  \int \ve^\beta \left| \Delta_\zeta R_x(x) \right|  \bigg| \log \bigg( 1 + \frac{ (\Delta_\zeta g(x))^2}{\zeta^2} \bigg) \bigg| \diff{\zeta} \\
& \hspace{1in}  + \int \ve \left| \Delta_\zeta V_x(x) \right| \bigg|  \log \bigg( 1 + \frac{ (\Delta_\zeta g(x))^2}{\zeta^2} \bigg) \diff{\zeta} -  \log \bigg( 1 + \frac{ (\ve\Delta_\zeta V(x))^2}{\zeta^2} \bigg) \bigg| \diff{\zeta} \bigg\}^2 \diff{x}\\
&\leq C \int \left[ \ve^{\beta + 1} |\Delta_\zeta R_x(x)| + \ve^{\beta + 1} | \Delta_\zeta V_x(x) | \right]^2 \diff{x},
\end{align*}
where, to obtain the last line, we use the mean value theorem on the logarithm and split the $\zeta$-integral into $|\zeta| \geq 1$ and $|\zeta| < 1$ as before.

Using the Leibniz rule and the bootstrap assumption \eqref{btsp}, we can also show that for $m \in \N \cup \{0\}$
\begin{align}
\label{leibniz}
\left\|\noncub[\ve^\beta R + \ve V] - \noncub[\ve V] ] \right\|_{H^m} \leq C \ve^{\beta + 1} \|R\|_{H^{m + 1}} + C \ve^{\beta + 1}.
\end{align}
It then follows that
\begin{align*}
J_3 + J_4 \leq C\ve^2 (1 + E).
\end{align*}

Finally, we consider $J_5 + J_6$
\begin{align*}
J_5 + J_6 & = \rho \int \partial^n \hilbert[R] \hilbert\left[\noncub[\ve^\beta R + \ve V] - \noncub[\ve V]\right]\partial^{n+1} \hilbert[R] \diff{x} \\
&~ \quad + \rho \int \hilbert[R] \partial^n \hilbert\left[\noncub[\ve^\beta R + \ve V] - \noncub[\ve V] \right]\partial^{n+1} \hilbert[R] \diff{x}\\
&~ \quad - \rho \int \hilbert[R] \partial^{n+1} \hilbert[R] \partial^n \hilbert\left[ \noncub[\ve^\beta R + \ve V] - \noncub[\ve V] \right]  \diff{x}\\
&~ \quad - (n + 1) \rho \int \partial \hilbert[R] \partial^n \hilbert[R] \partial^n \hilbert\left[ \noncub[\ve^\beta R + \ve V] - \noncub[\ve V] \right]  \diff{x} + \mathscr{R}\\
& = - \frac{\rho}{2} \int |\partial^n \hilbert[R]|^2 \partial \hilbert\left[\noncub[\ve^\beta R + \ve V] - \noncub[\ve V]\right] \diff{x} \\
&~ \quad - (n + 1) \rho \int \partial \hilbert[R] \partial^n \hilbert[R] \partial^n \hilbert\left[ \noncub[\ve^\beta R + \ve V] - \noncub[\ve V] \right]  \diff{x} + \mathscr{R}.
\end{align*}

For the first term on the right-hand side, we can use H\"{o}lder's inequality, Sobolev embedding, estimates \eqref{leibniz}, and the bootstrap assumption \eqref{btsp} to show that
\begin{align*}
& \bigg|\int |\partial^n \hilbert[R]|^2 \partial \hilbert\left[\noncub[\ve^\beta R + \ve V] - \noncub[\ve V]\right] \diff{x}\bigg|\\
\leq ~& C \|R\|_{H^n}^2 \bigg\|\partial \hilbert\left[\noncub[\ve^\beta R + \ve V] - \noncub[\ve V]\right]\bigg\|_{L^\infty}\\
\leq ~& C \ve^{\beta + 1} E^{3 / 2} \leq C \ve^2 E.
\end{align*}

For the remaining term, we use similar commutator estimates to the ones for $J_3 + J_4$, but distribute derivatives differently, to get
\begin{align*}
& \int \partial \hilbert[R] \partial^n \hilbert[R] \partial^n \hilbert[ \noncub[\ve^\beta R + \ve V] - \noncub[\ve V] ]  \diff{x} \\
\leq ~& C \|\partial^n R\|_{L^2} \|\partial^{n - 1} \hilbert[R]\|_{L^\infty} \left\|\noncub[\ve^\beta R + \ve V] - \noncub[\ve V] \right\|_{H^2}\\
\leq ~& C \beta^{\beta + 1} \|R\|_{H^n}^2 \|R\|_{H^3} \leq C \beta^{\beta + 1} E^{3 / 2} \leq C \ve^2 E,
\end{align*}
where in the last line follows from Sobolev embeddings, estimates \eqref{leibniz}, and the bootstrap assumption \eqref{btsp}.

\subsection{Energy estimates and enhanced lifespan for $R$}

Using the estimates for $I_1$--$I_{13}$ and $J_1$--$J_6$ in \eqref{dtE1}, we find, under the bootstrap assumption \eqref{btsp}, that
\[
\frac{\diff{E}}{\diff{t}} \leq C \ve^2 (1 + E^{1 / 2} + E) \leq C \ve^2 (1+E).
\]
We then get from Gronwall's inequality that
\begin{align}
\sup_{t \in [0, T / \ve^2]} E(t) &\leq C \left(E(0) + T\right) e^{C T},
\label{eng_gron}
\end{align}
so the bootstrap \eqref{btsp} is closed if $\ve>0$ is sufficiently small that
\[
C \left(E(0) + T\right) e^{C T} \leq \frac{1}{\ve^2}.
\]
Since $\|R(\cdot,0)\|_{H^n}$, and therefore $E(0)$, are bounded independently of $\ve$ for all sufficiently small $\ve$, the energy estimate \eqref{engest} follows from \eqref{eng_gron}.

These energy estimates assume additional smoothness on $R$. However, by continuous dependence of the Cauchy problem for \eqref{cde'}, we can approximate $\vp$ by smooth solutions $\vp^\nu$, carry out the energy estimates on $R^\nu = \ve^{-\beta}(\vp^\nu - \ve V)$, and take the limit as $\vp^\nu \to \vp$ in $C_t H^n_x$.
An alternative argument would be to use the \emph{a priori} estimates for $R$ to directly construct solutions of the error equation \eqref{err}.

We then see  from the boundedness of $\|R\|_{H^n}$ and $\|V\|_{H^n}$ that $\|\vp\|_{H^n}$ remains bounded, so it follows from local existence and uniqueness for \eqref{cde'} that
we can extend $\vp$ to a solution $\vp \in C([0, T / \ve^2]; H^n(\R))$ with $\vp(\cdot, 0) = \vp_0$. Moreover, the estimates \eqref{engest1} and \eqref{diff'} imply that this solution satisfies \eqref{solest}, which completes the proof of Theorem~\ref{thm:euler-cub}.

\appendix

\section{Contour dynamics for Euler vorticity fronts}
\label{sec:contour}

In this appendix, we will use contour dynamics to derive equation \eqref{cde} for $\vp(x,t)$, following the methods
used in \cite{HSZ19pb,HSZ20} for SQG and GSQG fronts.

The streamfunction-vorticity formulation for the velocity $\u(\x,t) = \left(u(x,y,t),v(x,y,t)\right)$ with $\x=(x,y)$ in the two-dimensional incompressible Euler equations is \cite{MB02}
\[
\alpha_t + \u\cdot\nabla \alpha = 0,\qquad
\u = \nabla^\perp \psi,\quad  -\Delta \psi = \alpha,\qquad \nabla^\perp = (-\partial_y,\partial_x),
\]
where $\psi(\x,t)$ is the streamfunction and it is convenient to use the negative vorticity $\alpha(\x,t)$.

For Euler front solutions with piecewise constant vorticities $-\alpha_+ \ne -\alpha_-$ that jump across $y=\vp(x,t)$ and approach linear shear
flows as $y\to \pm \infty$, we have
\begin{align}
\label{euler_front}
\alpha(\x,t) &= \begin{cases}\alpha_+ & \text{if $\ y > \vp(x,t)$},\\ \alpha_- & \text{if $y < \vp(x,t)$},\end{cases}
\qquad\u(\x,t) = \left(\alpha_\pm y,0\right) + o(1)\quad \text{as $y\to \pm \infty$}.
\end{align}
We will assume that $\vp$ satisfies the following conditions on a time interval $0\le t \le T$ with $T>0$:
\begin{align}\begin{split}
&\text{(i) $\varphi(\cdot,t) \in C^{1,\gamma}(\R)$ for some $\gamma > 0$;}\\
&\text{(ii) $\varphi_x(x,t) = O(|x|^{-(1+\delta)})$ as $|x|\to \infty$ for some $\delta > 0$;}\\
&\text{(iii) $\lim_{|x|\to \infty} \vp(x,t) = c$.}
\end{split}\label{asphi}\end{align}
In that case, the integrals below converge.

For any $h \in \R$, we denote the negative vorticity and velocity of a planar shear flow for a vorticity front located at $y=h$ by
\begin{align}
\label{defutilde1}
\begin{split}
&\tilde{\alpha}_h(y) = \begin{cases}\alpha_+ & \text{if $\ y > h$},\\ \alpha_- & \text{if $y < h$},\end{cases}
\qquad \tilde{\u}_h(y) = \left(\tilde{u}_h(y), 0 \right),
\\
&\tilde{u}_h(y) = \frac{1}{2}\Xi y + \frac{1}{2}\Theta |y - h|,
\\
&\Xi = \alpha_+ + \alpha_-,\qquad
\Theta = \alpha_+ - \alpha_-.
\end{split}
\end{align}
We then decompose the front solution \eqref{euler_front} as the sum of a shear flow and a perturbation
\begin{align}
\begin{split}
\alpha(\x, t) &= \tilde{\alpha}_h(\x) + \alpha^*_h(\x, t),
\qquad
\u(\x, t) = \tilde{\u}_h(y) + \u^*_h(\x, t),
\\
\alpha^*_h(\x,t) &= \begin{cases} - \Theta &\text{if $h < y < \vp(x, t)$},\\
\Theta & \text{if $h > y > \vp(x, t)$},\\
0 &\text{otherwise}.\end{cases}
\end{split}
\label{udecom}
\end{align}
with
\begin{align}
\label{u*h}
\u^*_h = \nabla^\perp \psi^*_h, \qquad   -\Delta\psi^*_h  = \alpha^*_h.
\end{align}

We will use the following orientations for the unit tangent vectors on the front and the line $y=h$:
\begin{align}
\label{deft}
\t(\x, t) = \frac{(1, \vp_{x}(x, t))}{\sqrt{1 + \vp_{x}^2(x, t)}}\quad \text{on $y = \vp(x, t)$},\qquad
 \t(\x, t) = (-1, 0) \quad \text{ on $y = h$}.
\end{align}

In the next two sections, we use two different choices of the parameter $h$ to derive \eqref{cde}.
In the first section, we take $h=c$ to be the limiting displacement of the front in \eqref{asphi},
which has the advantage that the standard potential representation for $\u^*_c$ converges under mild additional assumptions on $\vp$,
but the disadvantage that the line $y=c$ may intersect the front $y=\vp(x,t)$. In the second section, we choose
$h < \inf_{x\in \R} \vp(x,t)$, which has the advantage that the line $y=h$ does not intersect the front $y=\vp(x,t)$,
but the disadvantage that we have to modify the standard potential representation to get a convergent integral for $\u^*_h$.

\subsection{Contour dynamics equation I}

We make the choice $h = c$ in \eqref{udecom}, where $c$ is the far-field limit of the function $\vp$ given in \eqref{asphi}.
Since $\vp(\cdot,t)$ is continuous, the set $\{x\in \R: \vp(x, t) \ne c\}$ is open, and, by the structure of open sets in $\R$, it is the disjoint union of countably many open intervals. We denote these open intervals by $I_n=(a_n, b_n)$, with $-\infty \leq a_n<b_n\leq a_{n+1}<b_{n+1} \leq \infty$, $n\in \Z$. Then the set $\Omega^*(t) = \supp \alpha^*_c(\cdot, t)$ can be written as
\[
\Omega^*(t) = \bigcup_{n\in\Z}~\Omega^*_n(t),
\]
where each $\Omega^*_n(t)$ has one of the forms
\begin{align*}
&\{(x, y)\in \R^2 :  \text{$x \in I_n$, $c < y < \vp(x, t)$}\}\qquad \text{with $\alpha_c^*(\x, t) = -\Theta$},
\\
&\{(x, y)\in \R^2 :  \text{$x \in I_n$, $c > y > \vp(x, t)$}\}\qquad \text{with $\alpha_c^*(\x, t) = \Theta$}.
\end{align*}

Using the Biot-Savart law, we can express the velocity perturbation as
\begin{align*}
\u_c^*(\x,t) &=   \frac1{2 \pi}\iint_{\Omega^*(t)} \frac{(\x - \x')^\perp}{|\x - \x'|^2} \alpha^*_c(\x', t)\diff{\x'}
= \frac1{2\pi}\sum_{n\in \Z} \iint_{\Omega^*_n(t)} \frac{(\x - \x')^\perp}{|\x - \x'|^2} \alpha^*_c(\x', t) \diff{\x'},
\end{align*}
where $(x, y)^\perp = (-y, x)$. For each $\Omega^*_n(t)$, we apply Green's theorem to get
\[
 \iint_{\Omega^*_n(t)} \frac{(\x - \x')^\perp}{|\x - \x'|^2} \alpha_c^*(\x', t) \diff{\x'} =
\Theta \int_{\partial \Omega^*_n(t)} \t(\x', t) \log|\x - \x'| \diff{s(\x')},
\]
where the tangent vector $\t$ is defined as in \eqref{deft}.
We then find that
\begin{align*}
 &\iint_{\Omega^*_n(t)} \frac{(\x - \x')^\perp}{|\x - \x'|^2} \alpha_c^*(\x', t) \diff{\x'}\\
 &= \Theta \int_{I_n} (1,\vp_{x'}(x',t)) \log\left|\sqrt{(x-x')^2 + (\vp(x, t) - \vp(x', t))^2}\right| - (1, 0) \log\left|\sqrt{(x - x')^2+(\vp(x,t) - c)^2}\right| \diff{x'}.
\end{align*}
For unbounded components, a limiting procedure as in \cite{HSZ19pb} can be used, under a mild additional decay condition
that  $[\vp(x',t)-c]/x'$ is integrable for large $|x'|$, but we omit the details here. Summing these contributions, we get that
\begin{align}
\begin{split}
&\u_c^*(\x,t)
\\
&=\frac{\Theta}{2\pi} \int_\R (1,\vp_{x'}(x',t)) \log\left|\sqrt{(x-x')^2 + (\vp(x, t) - \vp(x', t))^2}\right| - (1, 0) \log\left|\sqrt{(x - x')^2+(\vp(x,t) - c)^2}\right| \diff{x'}.
\end{split}
\label{uc*}
\end{align}

Let $\x = (x,\varphi(x, t))$ be a point on the front and denote by
\begin{align}
\label{defn}
\n(\x,t) = \frac{1}{\sqrt{1+\varphi_x^2(x,t)}}(-\varphi_x(x,t),1)
\end{align}
the unit upward normal to the front. The front $y = \vp(x, t)$ moves with the upward normal velocity $\u \cdot \n$, namely $(0, \vp_t) \cdot \n = \u \cdot \n$,  so using \eqref{udecom}, we obtain that
\begin{align}
\label{cont0}
\vp_t(x, t) =\sqrt{1 + \vp_x^2(x, t)}\tilde{\u}_c(\x)\cdot \n(\x, t)+\sqrt{1 + \vp_x^2(x, t)}\u^*_c(\x, t)\cdot \n(\x, t).
\end{align}

From \eqref{uc*} and \eqref{defn}, we have
\begin{align*}
&\sqrt{1 + \vp_x^2(x, t)}\u^*_c(\x,t)\cdot\n(\x,t)\\
&\qquad=-\frac{\Theta}{4\pi}\vp_x(x, t)\int_{\R} \left\{\log\left[1 + \bigg(\frac{\vp(x, t) - \vp(x', t)}{x - x'}\bigg)^2\right] - \log\left[1 + \bigg(\frac{\vp(x, t) - c}{x - x'}\bigg)^2\right]\right\}\diff{x'}\\
&\qquad+\frac{\Theta}{4 \pi} \int_{\R} \vp_{x'}(x', t) \log\left[(x - x')^2 + (\vp(x, t) - \vp(x', t))^2\right]\diff{x'}\\
&\qquad = -\frac{\Theta}{4 \pi} \vp_x(x, t) \int_{\R} \log \left[1 + \bigg(\frac{\vp(x, t) - \vp(x', t)}{x - x'}\bigg)^2\right] \diff{x'} + \frac{\Theta}{2} \vp_x(x, t) |\vp(x, t) - c|\\
&\qquad+\frac{\Theta}{4 \pi} \int_{\R} \vp_{x'}(x', t) \log\left[(x - x')^2 + (\vp(x, t) - \vp(x', t))^2\right] \diff{x'},
\end{align*}
where we have used the identity
\[
\int_{\R} \log\bigg(1 + \frac{a^2}{x^2}\bigg) \diff{x}
= 2 \pi |a|.
\]
We observe that the above integrals converge thanks to \eqref{asphi}.

Moreover, from \eqref{defutilde1} and \eqref{defn}, we have
\begin{align*}
\sqrt{1 + \varphi_x^2(x, t)}\tilde{\u}_c(\x) \cdot \n(\x, t) & =
-\frac{\Xi}{2} \vp(x, t)\vp_x(x, t) - \frac{\Theta}{2}\vp_x(x, t) |\vp(x, t) - c|.
\end{align*}
Using these expressions in \eqref{cont0}, we get
\begin{align*}
\vp_t(x, t) = -\frac{\Xi}{4} \left(\vp^2(x, t)\right)_x - \frac{\Theta}{4 \pi} \vp_x(x, t) \int_{\R} \log\left[1+\bigg(\frac{\vp(x, t) - \vp(x', t)}{x - x'}\bigg)^2\right]\diff{x'} \\
+ \frac{\Theta}{4 \pi} \int_{\R} \vp_{x'}(x', t) \log\left[(x-x')^2+(\vp(x,t)-\vp(x',t))^2\right]\diff x'.
\end{align*}
Then, using the identity
\begin{align}
\label{lin_hilb}
\int_{\R} \vp_{x'}(x', t) \log(x - x')^2 \diff{x'} =  2\,\pv \int_{\R} \frac{\vp(x', t)-c}{x - x'} \diff{x'} = 2\pi\hilbert[\vp-c](x, t),
\end{align}
and making the substitution $x' = x + \zeta$, we find that $\vp$ satisfies
\begin{align}
\label{dim_cde}
\begin{split}
&\vp_t(x, t) + \frac{\Xi}{4} \px \left[\vp^2(x, t)\right] + \frac{\Theta}{4 \pi} \int_\R \left[\vp_x(x, t) - \vp_x(x + \zeta, t)\right] \log\bigg[1 + \frac{[\vp(x, t) - \vp(x + \zeta, t)]^2}{|\zeta|^2}\bigg] \diff{\zeta}
\\
&\qquad = \frac{\Theta}{2}\hilbert[\vp-c](x, t).
\end{split}
\end{align}
Nondimensionalizing the time variable by $t\mapsto \Theta t/2$ and setting $c=0$ without loss of generality, we obtain \eqref{cde}
with $m = \Xi/\Theta$. Equation \eqref{dim_cde} agrees with previous results in \cite{HS18} for the cubic front equation in the symmetric case with $\Xi=0$.

We remark that the corresponding dimensional version of the Burgers-Hilbert equation \eqref{bh} is
\begin{align*}
& u_t + \left(\frac{\sqrt{\Xi^2 + \Theta^2}}{4} u^2\right)_x = \frac{\Theta }{2}\hilbert[u].
\end{align*}

\subsection{Contour dynamics equation II}
We choose $h \in \R$ such that
\[
h < \inf \{\vp(x, t) : (x, t) \in \R \times [0, T]\}.
\]
The resulting $\alpha^*_{h}$ is then
\begin{align*}
\alpha^*_{h}(\x,t) = \begin{cases} -\Theta &\text{if $h < y < \vp(x,t)$}\\
 0 &\text{otherwise}\end{cases}
 \end{align*}
We denote the support of $\alpha^*_{h}(\cdot, t)$ by
\[
\Omega^*(t) = \left\{(x,y) \in \R^2 : h < y < \vp(x,t)\right\}.
\]

This choice of $h$ guarantees that the front $y = \vp(x, t)$ does not intersect with the artificial front $y = h$. However, the velocity integral using the usual Biot-Savart law does not converge. We therefore modify the Biot-Savart law by using
a potential that vanishes at a fixed point $\x_0 = (x_0,y_0)$, which can be chosen outside $\Omega_*$ for convenience,  rather than at infinity. A class of solutions of \eqref{u*h} for $\u^*_{h}$ then has the Green's function representation
\begin{equation}
\label{defu*}
\u^*_{h}(\x, t) =  \frac{\Theta}{2\pi} \int_{\Omega^*(t)}
\left\{\frac{(\x - \x')^\perp}{|\x - \x'|^2} - \frac{(\x_0 - \x')^\perp}{|\x_0 - \x'|^2}\right\} \diff{\x'}
+ \bar{\u}(t),
\end{equation}
where $\bar{\u}(t)$ is an arbitrary spatially uniform velocity.  We will choose $\bar{\u}(t)$
so that $\u(\x,t)$ has the asymptotic behavior in \eqref{euler_front} as $|y| \to \infty$.
The integral in \eqref{defu*} converges absolutely, since, if $\x' = (x', y')$, then the integrand is $O\left(|x'|^{-2}\right)$ as $|x'|\to \infty$ and compactly supported in $y'$.

We remark that the corresponding integral representation of $\u^*_{h}$ using the generalized Biot-Savart law in the SQG and GSQG equations converges absolutely, so it is not necessary to modify the standard generalized Biot-Savart kernel in that case \cite{HSZ19pb, HSZ20}.

Writing
\[
\frac{(\x - \x')^\perp}{|\x - \x'|^2} - \frac{(\x_0 - \x')^\perp}{|\x_0 - \x'|^2}
= -\nabla_{\x'}^\perp\left\{\log|\x-\x'| - \log|\x_0-\x'|\right\},
\]
applying Green's theorem in \eqref{u*h} on a truncated region with $|x-x'| < \lambda$ (as in \cite{HSZ19pb}), and taking the limit $\lambda\to\infty$, we get that
\begin{align}
\label{u*}
\u^*_{h}(\x,t) = \frac{\Theta}{2\pi}\int_{\partial \Omega^*(t)} \t(\x',t)\left\{\log|\x-\x'| - \log|\x_0-\x'|\right\}\diff{s(\x')}
+ \bar{\u}(t),
\end{align}
where $\t$ is  the negatively oriented unit tangent vector on $\partial \Omega^*$ defined as in \eqref{deft}.

If $\bar{\u} = (\bar{u},\bar{v})$,  then the component form of \eqref{u*} is
\begin{align*}
\begin{split}
\u^*_{h}(\x,t) &= \left(u^*(x,y,t),  v^*(x,y,t)\right)
\\
u^*_{h}(x,y,t) &=\frac{\Theta}{4\pi} \int_{\R} \bigg\{\log\bigg[\frac{(x - x')^2 + (y - \vp(x', t))^2}{(x-x')^2 + (y -h)^2}\bigg] - \log\bigg[\frac{(x_0-x')^2 + (y_0-\vp(x', t))^2}{(x_0-x')^2 + (y_0 -h)^2}\bigg] \bigg\} \diff{x'} + \bar{u}(t),
\\
v^*_{h}(x.y,t) &= \frac{\Theta}{4\pi}\int_{\R} \log\bigg[\frac{(x - x')^2 + (y - \vp(x', t))^2}{(x_0 - x')^2 + (y_0 - \vp(x', t))^2}\bigg] \vp_{x'}(x', t)\diff{x'}
+ \bar{v}(t).
\end{split}
\end{align*}
The integral for $u^*_{h}$ converges since the integrand is $O(|x'|^{-2})$ as $|x'|\to \infty$, while the integral for $v^*_{h}$ converges since $\vp_{x'}(x',t) = O(|x'|^{-(1+\delta)})$ as $|x'|\to \infty$.

Since $\vp(x,t) \to c$ as $|x|\to\infty$, we have as $|y|\to \infty$ that
\begin{align*}
\int_{\R} \log\bigg[\frac{(x - x')^2 + (y - \vp(x', t))^2}{(x - x')^2 + (y -h)^2}\bigg] \diff{x'}
&= |y|\int_{\R} \log\bigg[\frac{\eta^2 + (1-\vp(x+y\eta,t)/y)^2}{\eta^2 + (1 -h/y)^2}\bigg] \diff{\eta}
\\
&= |y| \int_{\R} \log\bigg[\frac{\eta^2 + (1-c/y)^2}{\eta^2 + (1 -h/y)^2}\bigg] \diff{\eta} + o(1)
\\
&=- 2(c+h)\sgn y \int_{\R } \frac{1}{1+\eta^2} \diff{\eta} + o(1)
\\
&= - 2\pi (c+h)\sgn y + o(1).
\end{align*}
We also have that
\begin{align*}
& \int_{\R} \log\left[(x-x')^2 + (y - \vp(x', t))^2\right] \vp_{x'}(x', t) \diff{x'}\\
=~& \log|y| \int_{\R}  \vp_{x'}(x', t) \diff{x'}
+ \int_{\R} \log\bigg[\bigg(\frac{x-x'}{y}\bigg)^2 + \left(1-\frac{\vp(x', t)}{y}\right)^2\bigg] \vp_{x'}(x', t) \diff{x'}
\\
= ~& \log|y| \int_{\R}  \vp_{x'}(x', t) \diff{x'} + o(1)\qquad \text{as $|y|\to \infty$},
\end{align*}
so the $y$-component of the velocity approaches zero if  $\int \vp_{x'}(x', t) \diff{x'} = 0$, which is the case if
$\vp$ satisfies \eqref{asphi}.

Under the assumptions in \eqref{asphi}, it follows that the velocity perturbations have the asymptotic behavior as $|y|\to\infty$
\begin{align*}
u^*_{h}(x,y,t) &= -\frac{1}{2} \Theta (c+h)\sgn y + \bar{u}(t)-u_\infty(t) + o(1),
\\
v^*_{h}(x,y,t) &=  \bar{v}(t) - v_\infty(t) + o(1) ,
\\
{u}_\infty(t) &= \frac{\Theta}{4\pi} \int_{\R} \log\bigg[\frac{(x_0 - x')^2 + (y_0 - \vp(x', t))^2}{(x_0 - x')^2 + (y_0 -h)^2}\bigg] \diff{x'},
\\
{v}_\infty(t) &= \frac{\Theta}{4\pi}\int_{\R} \log\left[(x_0 - x')^2 + (y_0 - \vp(x', t))^2\right] \vp_{x'}(x', t) \diff{x'}.
\end{align*}

We choose $\bar{\u} = ({u}_\infty,{v}_\infty)$ in \eqref{u*}, in which case, using the integral
\begin{equation}
\int_{\R} \log\bigg[\frac{x^2 + a^2}{x^2+b^2}\bigg] \diff{x} = 2\pi\left(|a| - |b|\right)
\label{log_int}
\end{equation}
to replace $h$ by $c$ in $u$,
we find that the full velocity field $\u = (u, v)$ can be written as
\begin{align*}
u(x, y, t) &=  \frac{\Theta}{4\pi} \int_{\R} \log\bigg[\frac{(x - x')^2 + (y - \vp(x', t))^2}{(x - x')^2 + (y -c)^2}\bigg] \diff{x'}
+ \frac{1}{2} \Xi y + \frac{1}{2} \Theta |y -c|,
\\
v(x, y, t) &=  \frac{\Theta}{4 \pi}
\int_{\R} \log\bigg[(x - x')^2 + (y - \vp(x', t))^2\bigg] \vp_{x'}(x', t) \diff{x'}.
\end{align*}
This velocity has the far-field behavior in \eqref{euler_front} as $|y| \to \infty$.

If $\x = \left(x, \varphi(x, t)\right)$ is a point on the front and $\vp = \vp(x, t)$, then
\begin{align*}
u(x,\vp,t) &=  \frac{\Theta}{4\pi}\int_{\R} \log\bigg[\frac{(x - x')^2 + (\vp(x, t) - \vp(x', t))^2}{(x - x')^2 + (\vp(x, t) - c)^2}\bigg] \diff{x'}
+ \frac{1}{2} \Xi \vp + \frac{1}{2} \Theta |\vp(x, t) - c|,
\\
v(x,\vp,t) &=  \frac{\Theta}{4 \pi}
\int_{\R} \log\bigg[(x - x')^2 + (\vp(x, t) - \vp(x', t))^2\bigg] \vp_{x'}(x', t) \diff{x'}.
\end{align*}
Imposing the condition that the front $y = \vp(x, t)$ moves with the upward normal velocity $\u \cdot \n$,  namely $(0, \vp_t) \cdot \n = \u \cdot \n$, we get that
\begin{equation}
\vp_t + \frac{1}{2}\left(\Xi \vp + \Theta |\vp-c| + \Theta F_1\right) \vp_x  = \frac{1}{2}\Theta F_2,
\label{temp_vp}
\end{equation}
where
\begin{align*}
F_1(x, t) &= \frac{1}{2\pi}\int_{\R} \log\bigg[\frac{(x - x')^2 + (\vp(x, t) - \vp(x', t))^2}{(x - x')^2 + (\vp(x, t) - c)^2}\bigg] \diff{x'},\\
 F_2(x,t) &= \frac{1}{2 \pi} \int_{\R} \vp_{x'}(x', t) \log\bigg[(x - x')^2 + (\vp(x, t) - \vp(x', t))^2\bigg]\diff{x'}.
\end{align*}

Using \eqref{lin_hilb} and \eqref{log_int}, we can write
\begin{align*}
F_1(x, t)
&= \frac{1}{2\pi} \int_{\R} \log\bigg[\frac{(x - x')^2 + (\vp(x, t) - \vp(x', t))^2}{(x - x')^2}\bigg] \diff{x'} - |\vp(x, t) - c|,
\\
F_2(x,t)
& = \frac{1}{2\pi} \int_{\R} \vp_{x'}(x', t)\log\bigg[\frac{(x - x')^2 + (\vp(x, t) - \vp(x', t))^2}{(x - x')^2}\bigg] \diff{x'}
+ \hilbert[\vp - c](x,t).
\end{align*}
Using these expressions in \eqref{temp_vp}, simplifying the result,  and substituting $x' = x + \zeta$, we get \eqref{dim_cde} as before.

\section{A lemma}

In this appendix, we show that the removal of the linear evolution $e^{t\hilbert}$ from \eqref{weq} leads to \eqref{cub}.
\begin{lemma}
\label{lem:vwtrans}
Suppose that $v(x,\tau)$ and $w(x,t)$ are related by \eqref{defw}. Then $v(x,\tau)$ satisfies \eqref{cub} if and only if
$w(x,t)$ satisfies \eqref{weq}.
\end{lemma}

\begin{proof}
We write \eqref{weq} as
\[
w_t + \frac{m^2+1}{6} \px M(w,w,w) = \hilbert[w], \quad M(w,w,w) = 3w^2 |\px| w - 3w |\px| w^2 + |\px| w^3,
\]
where $M$ is a symmetric trilinear operator with
\begin{align*}
M(e^{ikx}, e^{i\xi x}, e^{i\eta x}) &= m(k,\xi,\eta) e^{i(k+\xi+\eta)x},
\\
m(k,\xi,\eta) &= |k| + |\xi| + |\eta| - |k+\xi| - |\xi+\eta| - |k+\eta| + |k+\xi+\eta|.
\end{align*}
Using \eqref{defw} in \eqref{weq}, we find that
\[
v_\tau + \frac{m^2+1}{6} e^{-t\hilbert} \px M(e^{t\hilbert}v, e^{t\hilbert}v, e^{t\hilbert}v) = 0,
\]
so we just have to show that $M(e^{t\hilbert}v, e^{t\hilbert}v, e^{t\hilbert}v) = e^{t\hilbert} M(v,v,v)$, but this follows from the
fact that
\[
\sgn k + \sgn \xi + \sgn \eta = \sgn(k+\xi+\eta)
\]
whenever $m(k,\xi,\eta) \ne 0$.
\end{proof}

\section{Simplification of a solvability condition}
\label{sec:simpsolv}

In this appendix, we provide the details of reducing the solvability condition for \eqref{ve3'21}, given by
\begin{align*}
\P\bigg[\Psi_\tau + \rho \left(\Psi \Psi_{10} + \Psi^* \Psi_{12}\right)_x + \frac{\sigma}{2} \px \bigg\{2 |\Psi|^2 |\px| \Psi + \Psi^2 |\px| \Psi^* - 2 \Psi |\px| |\Psi|^2 - \Psi^* |\px| \Psi^2 + |\px| (\Psi |\Psi|^2)\bigg\}\bigg] = 0,
\end{align*}
to \eqref{psieqn}.

Using $\P[\Psi] = \Psi$ and \eqref{ve2a}, we see that it suffices to show
\begin{align}
\label{solv1}
\P\bigg[\Psi \hilbert[|\Psi|^2]_x + \frac{i}{2} \Psi^* (\Psi^2)_x\bigg] = \P\left[i |\Psi|^2 \Psi_x + \Psi \hilbert[|\Psi|^2]_x\right],
\end{align}
and
\begin{align}
\label{solv2}
- \frac{1}{2} \P\bigg[2 |\Psi|^2 |\px| \Psi + \Psi^2 |\px| \Psi^* - 2 \Psi |\px| |\Psi|^2 - \Psi^* |\px| \Psi^2 + |\px| (\Psi |\Psi|^2)\bigg] = \P\left[i |\Psi|^2 \Psi_x + \Psi \hilbert[|\Psi|^2]_x\right].
\end{align}
Equation \eqref{solv1} is immediate, since $\frac{i}{2} \Psi^* (\Psi^2)_x = i |\Psi|^2 \Psi_x$.
To prove \eqref{solv2}, we  use the identities
\[
|\px| \Psi = -i \Psi_x,\qquad |\px| \Psi^* = i \Psi^*_x,\qquad |\px| =  \hilbert\px ,\qquad \P|\px| = -i\px
\]
and obtain that
\begin{align*}
- \P\left[|\Psi|^2 |\px| \Psi\right] & =  \P\left[i |\Psi|^2 \Psi_x\right],\\
- \frac{1}{2} \P\left[\Psi^2 |\px| \Psi^*\right] & = -\frac{1}{2} \P\left[i \Psi^2 \Psi^*_x\right],\\
\P\left[\Psi |\px| |\Psi|^2\right] & = \P\left[\Psi \hilbert[|\Psi|^2]_x\right],\\
\frac{1}{2} \P\left[\Psi^* |\px| \Psi^2\right] &  = - \P\left[i |\Psi|^2 \Psi_x\right],\\
- \frac{1}{2} \P\left[|\px| (\Psi |\Psi|^2)\right] & =  \P\left[i |\Psi|^2 \Psi_x\right] + \frac{1}{2} \P\left[i \Psi^2 \Psi^*_x\right].
\end{align*}
Summing these terms gives \eqref{solv2}.

\end{document}